\documentclass[11pt,reqno]{amsart}
\usepackage[a4paper]{anysize}\marginsize{2.5cm}{2.5cm}{2cm}{2cm}
\pdfpagewidth=\paperwidth \pdfpageheight=\paperheight
\allowdisplaybreaks
\usepackage{mlmodern}

\usepackage[english]{babel}
\usepackage[latin1]{inputenc}
\usepackage[T1]{fontenc}
\usepackage{amssymb,amsmath,amstext,amsfonts,amsthm,amscd,latexsym, mathrsfs}
\usepackage{mathtools}
\usepackage{verbatim}
\usepackage{relsize}
\usepackage[colorlinks=true, urlcolor=blue, linkcolor=blue, citecolor=blue]{hyperref}
\usepackage[percent]{overpic}
\usepackage{graphicx}
\usepackage{xcolor}
\usepackage{cprotect} 
\usepackage{tcolorbox}
\usepackage{fancyvrb}

\numberwithin{equation}{section}
\theoremstyle{plain}
\newtheorem{theorem}{Theorem}
\numberwithin{theorem}{section}
\newtheorem{proposition}[theorem]{Proposition}
\newtheorem{lemma}[theorem]{Lemma}
\newtheorem{corollary}[theorem]{Corollary}

\newtheorem{remark}[theorem]{Remark}

\theoremstyle{definition}
\newtheorem{definition}[theorem]{Definition}
\newtheorem{example}[theorem]{Example}

\newcommand{\red}[1]{\textcolor{red}{#1}}

\newcommand{\C}{{\mathbb C}}

\newcommand{\K}{{\mathbb K}}
\newcommand{\R}{{\mathbb R}}
\newcommand{\Z}{{\mathbb Z}}
\newcommand{\N}{{\mathbb N}}
\newcommand{\V}{{\mathbb V}}
\newcommand{\PP}{{\mathbb P}}

\def\be{{\boldsymbol{e}}}
\def\bi{{\boldsymbol{i}}}
\def\bj{{\boldsymbol{j}}}
\def\bm{{\boldsymbol{m}}}
\def\bn{{\boldsymbol{n}}}

\def\bd{{\boldsymbol{d}}}
\def\balpha{{\boldsymbol{\alpha}}}

\def\bone{{\boldsymbol{1}}}
\newcommand{\mT}{\mathsmaller{\mathsf{T}}}
\newcommand{\mR}{\mathsmaller{\mathbb{R}}}

\newcommand{\mK}{\mathsmaller{\mathbb{K}}}
\newcommand{\mBW}{\mathsmaller{\mathrm{BW}}}
\newcommand{\dist}{\operatorname{dist}}
\newcommand{\rank}{\operatorname{rank}}
\newcommand{\codim}{\operatorname{codim}}
\newcommand{\RRdeg}{\mathrm{RRdeg}}

\newcommand{\ke}{{\mathcal E}}
\newcommand{\kf}{{\mathcal F}}
\newcommand{\kg}{{\mathcal G}}

\newcommand{\kl}{{\mathcal L}}

\newcommand{\kn}{{\mathcal N}}
\newcommand{\ko}{{\mathcal O}}
\newcommand{\kp}{{\mathcal P}}

\newcommand{\kt}{{\mathcal T}}

\makeatletter
\@namedef{subjclassname@2020}{%
	\textup{2020} Mathematics Subject Classification}
\makeatother

\author{Flavio Salizzoni}
\author{Luca Sodomaco}
\address{Max Planck Institute for Mathematics in the Sciences, Leipzig, Germany}
\email{flavio.salizzoni@mis.mpg.de}
\email{luca.sodomaco@mis.mpg.de}
\author{Julian Weigert}
\address{Max Planck Institute for Mathematics in the Sciences, Leipzig, Germany}
\address{Mathematisches Institut, Universit\"at Leipzig, Augustusplatz 10, 04109 Leipzig, Germany}
\email{julian.weigert@mis.mpg.de}



\title[Nonlinear Rayleigh quotient optimization]{Nonlinear Rayleigh quotient optimization}

\date{}

\begin{document}

\begin{abstract}
Rayleigh quotient minimization deals with optimizing a quadratic homogeneous function over a sphere.
Its critical points correspond to the normalized eigenvectors of the symmetric matrix associated with the quadratic form.
In this paper, we consider a homogeneous polynomial objective function $f$ over a sphere, a projective algebraic variety $X$, and we study the {\em $X$-eigenpoints of $f$}, which are classes of critical points of $f$ constrained to the sphere and the affine cone over $X$.
The number of $X$-eigenpoints of a generic polynomial $f$ is the {\em Rayleigh-Ritz degree of $X$}.
This invariant is a version of the Euclidean distance degree of a Veronese embedding of $X$.
We provide concrete formulas in various scenarios, including those involving varieties of rank-one tensors.
\end{abstract}

\maketitle

\section{Introduction}

We study the critical points of the constrained minimization problem
\begin{equation}\label{eq: main optimization}
    \min_{\psi\in S^n}f(\psi)\quad\text{subject to $\psi\in C(X)$}\,,
\end{equation}
where $f$ is a homogeneous polynomial function of degree $\omega\ge 1$ defined over the unit sphere $S^n\subseteq\R^{n+1}$, and $C(X)\subseteq\R^{n+1}$ is the affine cone over a projective variety $X\subseteq\PP^n$.

If $\omega=2$, then $f(\psi)=\psi^\mT H\psi$ for some $(n+1)\times(n+1)$ real symmetric matrix $H$, and we can rephrase \eqref{eq: main optimization} as a {\em constrained Rayleigh quotient minimization problem}.
If additionally $X=\PP^n$, the critical points of \eqref{eq: main optimization} are the normalized eigenvectors of $H$.
The motivation to study the problem \eqref{eq: main optimization} for $X\subsetneq\PP^n$ comes, e.g., from the quantum many-body problem in quantum chemistry, condensed-matter physics, and quantum field theory, see \cite[Section 6.3]{vanderstraeten2019tangent}. In these applications, the main goal is not typically to find the complete spectrum of $H$, the Hamiltonian matrix, but rather to determine only the eigenvector corresponding to the minimum eigenvalue, the so-called {\em ground state}. Furthermore, it is assumed that the ground state belongs to a special algebraic variety, for example, a tensor network variety. Indeed, the use of tensor network structures in physical chemistry is a vast subject area known as density matrix renormalization group (DMRG). The articles \cite{szalay2015tensor,baiardi2020density} offer a detailed review.
The first significant case occurs when $X$ is defined by linear equations. This situation is commonly referred to in the literature as the constrained eigenvalue problem. The problem was studied first in \cite{golub1973modified,golub1983matrix}, and the name was coined in \cite{gander1989constrained}. A detailed overview is given in \cite{zhou2021linear}.
When $X$ is defined by quadratic equations of the form $\psi^TM_i\psi=0$ for some subset of symmetric matrices $M_i$, the problem is studied in \cite{karow2011values,prajapati2022optimizing}, under the assumption that $H$ is a Hermitian matrix. In particular, it is shown in \cite{sharma2015eigenvalue} that, for quadratic constraints, this problem is related to the so-called structured eigenvalue backward error problem, see also \cite{bora2014structured}.

If $\omega\ge 2$ and $X=\PP^n$, the critical points of \eqref{eq: main optimization} correspond to the {\em normalized eigenvectors} of $f$, which is identified with a {\em symmetric tensor} of degree $\omega$, see Remark \ref{rmk: normalized eigenvectors}. The notions of tensor eigenvalue and eigenvector were proposed independently by Lim and Qi in \cite{lim2005singular,qi2005eigenvalues}. There are different types of eigenvectors and eigenvalues in the literature, see \cite{qi2017spectral} for an extensive overview on the subject.

We define an {\em $X$-eigenpoint} of $f$ as any class $[\psi]\in\PP^n$ of a critical point $\psi$ of \eqref{eq: main optimization}. For a sufficiently generic polynomial $f$, the number of complex $X$-eigenpoints is finite and constant. We call this invariant {\em Rayleigh-Ritz degree (RR degree)} of $X$. The name was suggested by the authors of \cite{borovik2025numerical}, which studies critical points of \eqref{eq: main optimization} for tensor train varieties $X$. In Section \ref{sec: RR degree} we define the RR degree of $X$ and relate it to the {\em distance degree} of the degree-$\omega$ Veronese embedding of $X$, with respect to a Bombieri-Weyl inner product.
The classical {\em Euclidean distance degree} was introduced in \cite{draisma2016euclidean}, and for the notations used in this paper, we mainly follow \cite{dirocco2025osculating}.
Hence, studying the optimization problem \eqref{eq: main optimization} naturally fits into the context of Metric Algebraic Geometry \cite{breiding2024metric}, a branch of applied Algebraic Geometry that focuses on the study of distance and metric problems using algebraic tools.

In the following sections, we establish closed RR degree formulas for various choices of $X$. In Section \ref{sec: Porteous}, we apply Porteous' Theorem to derive closed RR degree formulas for complete intersections and images of sufficiently generic polynomial maps. Section \ref{sec: RR degrees general position} examines the case of an arbitrary nonsingular projective variety $X$ positioned in a sufficiently general manner. Finally, Section \ref{sec: RR degrees singular tuples} focuses on special Segre-Veronese embeddings. Notably, the critical points of equation \eqref{eq: main optimization} correspond to the singular vector tuples of a multihomogeneous polynomial associated with the given polynomial $f$. Additionally, we discuss real $X$-eigenvectors of real homogeneous polynomials and illustrate in Example \ref{ex: no six real} that the findings in \cite{kozhasov2018fully} do not extend to $X$-eigenvectors when $X$ is not a projective space.

\section{The Rayleigh-Ritz degree of a projective variety}\label{sec: RR degree}

We begin by establishing the general notations used throughout the paper. We denote by $\K$ an algebraically closed field of characteristic zero. In our applications $\K$ is either the field $\R$ or $\C$. We use the notation $\psi$ for a vector in $\K^{n+1}$. In coordinates, we consider $\psi=(\psi_0,\ldots,\psi_n)^\mT$ as a column vector. We denote by $\PP_{\mK}^n\coloneqq\PP(\K^{n+1})$ the $n$-dimensional projective space of lines passing through the origin of $\K^{n+1}$. In particular, given a vector $\psi\in\K^{n+1}$, we denote by $[\psi]=[\psi_0:\cdots:\psi_n]$ its equivalence class in $\PP_{\mK}^n$. We use $x=(x_0,\ldots,x_n)$ as unknowns on $\K^{n+1}$ and denote by $\psi^*\in(\K^{n+1})^*$ the linear form defined by $\psi^*(x)\coloneqq \psi_0x_0+\cdots+\psi_nx_n$. An algebraic variety over $\K$ is the solution set in $\K^{n+1}$ of a system of polynomial equations in $n+1$ variables with coefficients in $\K$. We indicate it in the form $X^\mK$, or often simply by $X$ when the ground field is $\K=\C$. We use the notation $\V(f)$ or $\V(A)$ to denote the zero locus of a polynomial $f\in\K[x]$ or of a subset $A\subseteq\K[x]$. When the polynomials that define an algebraic variety $X^\mK\subseteq\K^{n+1}$ are homogeneous, we can view $X^\mK$ as a subset of $\PP_{\mK}^n$. In this case, we say that $X^\mK$ is a projective variety in $\PP_{\mK}^n$. Sometimes we want to stress the difference between a projective variety $X^\mK\subseteq\PP_{\mK}^n$ and the corresponding algebraic set in $\K^{n+1}$: in fact, we denote by $C(X^\mK)$ the affine cone in $\K^{n+1}$ over $X^\mK$ cut out by the same polynomials defining $X^\mK$. Given a subset $A\subseteq\K^{n+1}$, its Zariski closure is denoted by $\overline{A}$, and similarly for subsets of $\PP_{\mK}^n$. Given a function $f\colon\K^{n+1}\to\K$ and an algebraic variety $X^\mK\subset\K^{n+1}$, we denote by $\mathrm{Crit}(f,X^\mK)$ the subset of $\K^{n+1}$ of nonsingular points of $X^\mK$ which are critical for the restriction $f|_{X^\mK}$. We denote by $S^n$ the standard unit sphere in $\C^{n+1}$ of equation $x_0^2+\cdots+x_n^2=1$.

\begin{definition}\label{def: RR degree}
Consider a variety $X\subseteq\PP^n$ and a positive integer $\omega$.
Given a polynomial function $f\colon\C^{n+1}\to\C$ of degree $\omega$, we say that $\psi\in\C^{n+1}$ is a {\em (normalized) $X$-eigenvector} of $f$ if $\psi\in\mathrm{Crit}(f,C(X)\cap S^n)$.
The {\em (normalized) $X$-eigenvalue} of $f$ associated with $\psi$ is $\lambda\coloneqq f(\psi)$. We also refer to $(\psi,\lambda)$ as a {\em (normalized) $X$-eigenpair} of $f$. Furthermore, we call any class $[\psi]\in\PP^n$ of a normalized $X$-eigenvector of $f$ an {\em $X$-eigenpoint} of $f$. 

The {\em Rayleigh-Ritz degree (or RR degree) of index $\omega$ of $X$} is the cardinality
\begin{equation}\label{eq: RR degree}
    \RRdeg_{\omega}(X) \coloneqq |\{[\psi]\in\PP^n \mid\text{$[\psi]$ is an $X$-eigenpoint of $f$}\}|
\end{equation}
for a generic polynomial function $f\colon\C^{n+1}\to\C$ of degree $\omega$.
\end{definition}

Observe that the number of complex critical points of $\eqref{eq: main optimization}$ is two times the RR degree of $X$ because $\{\psi,-\psi\}\subseteq\mathrm{Crit}(f,C(X)\cap S^n)$ for every $X$-eigenpoint $[\psi]$ of $f$.
In the following, we omit the term ``normalized'' when considering $X$-eigenvectors and $X$-eigenvalues. The number of $X$-eigenpoints depends on the choice of the polynomial function $f$, but it is constant in an open dense subset of polynomial functions. So, the notion of RR degree is well-defined. This is not a trivial fact, and it is a consequence of Corollary~\ref{corol: RR degree equals BW degree}.
We also point out that Definition \ref{def: RR degree} is motivated by the case $\omega=2$, where we recover an optimization problem as in \eqref{eq: main optimization}. The case of normalized $X$-eigenvectors for $X=\PP^n$ is discussed in the following remark.

\begin{remark}\label{rmk: normalized eigenvectors}
We refer to \cite[\S 2.2]{qi2017spectral} for more details. In the case $X=\PP^n$ with any positive integer $\omega$, given a homogeneous polynomial $f$ of degree $\omega$, we say that a nonzero vector $\psi\in\C^{n+1}$ is a {\em (normalized) eigenvector (or E-eigenvector) of $f$} if $\psi\in S^n$ and there exists $\lambda\in\C$ such that
\[
\frac{1}{\omega}\nabla f(\psi) = \lambda\,\psi\,.
\]
The value $\lambda=f(\psi)$ is the {\em (normalized) eigenvalue of $f$} associated with $\psi$.
The pair $(\psi,\lambda)$ is called a {\em normalized eigenpair of $f$}. Furthermore, we call any class $[\psi]\in\PP^n$ of an eigenvector of $f$ an {\em eigenpoint} of $f$. In particular, the eigenpoints of $f$ are the fixed points of the gradient map
\[
\nabla f\colon\PP^n\dashrightarrow\PP^n\,,\quad\nabla f([\psi])=\left[\frac{\partial f}{\partial x_0}(\psi):\cdots:\frac{\partial f}{\partial x_n}(\psi)\right]\,.
\]
\end{remark}

Note that both Definition \ref{def: RR degree} and Remark \ref{rmk: normalized eigenvectors} can be rewritten replacing $S^n$ with the unit sphere $S_q$ associated with any nondegenerate complex quadratic form $q$.
In this work, to study distance functions, we consider only positive-definite real quadratic forms.

As pointed out in the introduction, our first goal is to relate the RR degree of an algebraic variety $X$ to a certain distance degree of $X$. For this reason, we now outline the algebraic study of the distance function from a point restricted to an algebraic variety. Let $V^\mR$ be a real vector space equipped with an inner product $\langle\,,\rangle\colon V^\mR\times V^\mR\to\R$, namely a real positive-definite symmetric bilinear form. We denote by $q(x)\coloneqq\langle x,x\rangle$ the associated quadratic form. In our applications, $V^\mR$ is either $\R^{n+1}$ or the space of homogeneous polynomials of degree $\omega$ with real coefficients. Given a vector $u\in V^\mR$, consider the squared distance function from $u$
\[
\dist_{q,u}^2\colon V^\mR\to\R\,,\quad\dist_{q,u}^2(x)\coloneqq q(x-u)\,.
\]
Given a subset $A\subseteq V^\mR$, typically not containing the data point $u\in V^\mR$, a classical optimization problem is finding the closest point on $A$ to $u$:
\begin{equation}\label{eq: general minimization problem}
\min_{x\in A}\dist_{q,u}^2(x)\,.
\end{equation}
We are interested in the case when $A$ is a real algebraic variety $X^\mR\subseteq V^\mR$. The global minimum of the restriction $\dist_{q,u}^2|_{X^\mR}$ may be a singular point of $X^\mR$: for example, given any data point $u=(u_1,0)\in\R^2$ with $u_1<0$ and the cuspidal cubic $X^\mR=\V(x_1^3-x_2^2)$, then the closest point on $X^\mR$ to $u$ with respect to the Euclidean quadratic form is the cusp $(0,0)\in X^\mR$.
In this paper, we focus on the subset of nonsingular local extrema on $X^\mR$ of the restriction $\dist_{q,u}^2|_{X^\mR}$.
To utilize tools from Intersection Theory in Complex Algebraic Geometry, we consider the complex vector space $V\coloneqq V^\mR\otimes\C$, the Zariski closure $\overline{X^\mR}\subseteq V$, and we determine the subset of critical points on $\overline{X^\mR}$ of the complex-valued polynomial objective function $\dist_{q,u}^2$, according to the following definition.

\begin{definition}\label{def: critical}
Let $V^\mR$ be a real vector space equipped with an inner product $\langle\,,\rangle$ with associated quadratic form $q$. Consider a real algebraic variety $X^\mR\subseteq V^\mR$ and define $X\coloneqq\overline{X^\mR}\subseteq V$. For a given $u\in V$, a point $x\in X$ is {\em critical} for the complex-valued polynomial function $\dist_{q,u}^2\colon V\to\C$ if $x$ is a nonsingular point of $X$ and the vector $x-u$ is in the {\em normal space of $X$ at $x$}:
\[
N_x X \coloneqq \{ y\in V \mid\text{$\langle y,v\rangle=0$ for all $v\in T_xX$}\}\,,
\]
where $T_xX\subseteq V$ is the tangent space of $X$ at $x$. We denote by $\mathrm{Crit}(\dist_{q,u}^2,X)\subseteq V$ the subset of critical points of $\dist_{q,u}^2$ on $X$.
\end{definition}

Applying \cite[Lemma 2.1]{draisma2016euclidean}, there exists a Zariski open subset of complex data points $u\in V$ such that $\mathrm{Crit}(\dist_{q,u}^2,X)$ consists of finitely many points, and its cardinality is the maximum one. The cardinality of this subset is the {\em distance degree of $X$} with respect to $q$. We denote this invariant by $\mathrm{DD}(X,Q)$, where $Q$ is the quadric hypersurface in projective space $\PP(V)$ defined by the complex zeros of the polynomial $q$. In particular, the subset of local minima of $\dist_{q,u}^2|_{X^\mR}$ that are nonsingular points of $X^\mR$ is a subset of $\mathrm{Crit}(\dist_{q,u}^2,X)$. For this reason, we say that the invariant $\mathrm{DD}(X,Q)$ measures the {\em algebraic complexity} of solving the minimization problem \eqref{eq: general minimization problem} for $A=X^\mR$.
The value of $\mathrm{DD}(X,Q)$ strongly depends on the intersection between the projective varieties $X$ and $Q$. In the following, we say that $X$ is {\em in general position with respect to $Q$} if $X$ and $Q$ intersect transversally.
More generally, following the discussion in \cite[Remark 2.3]{kozhasov2023minimal}, the distance degree $\mathrm{DD}(X,Q)$ is a well-defined algebraic invariant for any complex algebraic variety $X\subseteq V$ and for any nonsingular complex quadric hypersurface $Q\subseteq\PP(V)$, provided that $X\not\subseteq C(Q)$. Furthermore, given a projective variety $Z\subseteq\PP(V)$ and a nonsingular complex quadric hypersurface with $Z\not\subseteq Q$, we define the distance degree of the pair $(Z,Q)$ as $\mathrm{DD}(Z,Q)\coloneqq\mathrm{DD}(C(Z),Q)$.

In this work, we are especially interested in distance degrees of algebraic varieties in spaces of symmetric tensors.
Let $\omega$ be a positive integer. A {\em symmetric tensor of order $\omega$ in $n+1$ variables over $\K$} is an element of the tensor space $\mathrm{Sym}^{\omega}(\K^{n+1})^*$, which we identify with the space $\K[x]_{\omega}=\K[x_0,\ldots,x_n]_{\omega}$ of homogeneous polynomials of degree $\omega$ in $n+1$ variables. Consider the Veronese map
\begin{equation}
\nu_{\omega}\colon\PP_{\mK}^n\to\PP(\K[x]_{\omega})\,,\quad\nu_{\omega}([\psi]) \coloneqq [(\psi^*)^{\omega}]\,.
\end{equation}
The image of $\nu_{\omega}$ is the Veronese variety $V_{\omega}^\mK\subseteq\PP(\K[x]_{\omega})$. Its elements are also regarded as {\em rank at most one symmetric tensors} over $\K$. When $\K=\R$, it is natural to equip the real space $\R[x]_{\omega}$ with an inner product induced, via the map $\nu_{\omega}$, by a fixed inner product in $\R^{n+1}$. This motivates a classical definition.

\begin{definition}\label{def: Bombieri-Weyl inner product}
Let $\langle\,,\rangle$ be an inner product on $\R^{n+1}$ and $\omega$ a positive integer. Fix an orthonormal basis $\{e_0,\ldots,e_n\}$ of $\R^{n+1}$ with respect to $\langle\,,\rangle$, and define $x_i=e_i^*\in(\R^{n+1})^*$ for all $i\in\{0,\ldots,n\}$.
The {\em Bombieri-Weyl inner product} associated to $\langle\,,\rangle$ and $\omega$ is the inner product $\langle\,,\rangle_\mBW$ on $\R[x]_{\omega}$ such that
\begin{equation}\label{eq: orthonormal basis BW}
\left\{\binom{\omega}{\alpha}^{\frac{1}{2}}x^\alpha\right\}_{|\alpha|=\omega} = \left\{\left(\frac{\omega!}{\alpha_0!\cdots\alpha_n!}\right)^{\frac{1}{2}}x_0^{\alpha_0}\cdots x_n^{\alpha_n}\right\}_{|\alpha|=\omega}
\end{equation}
is an orthonormal basis of $\R[x]_{\omega}$. We also denote by $q_\mBW$ the Bombieri-Weyl quadratic form associated to $\langle\,,\rangle_\mBW$.
\end{definition}

The following identities are an almost immediate consequence of Definition \ref{def: Bombieri-Weyl inner product}:
\begin{equation}\label{eq: identities BW inner product}
\left\langle\prod_{j=1}^{\omega}\varphi_j^*,\prod_{j=1}^{\omega}\psi_j^*\right\rangle_\mBW = \frac{1}{\omega!}\sum_{\sigma\in\Sigma_{\omega}}\prod_{j=1}^{\omega} \langle \varphi_j,\psi_{\sigma(j)}\rangle\quad\forall\,\varphi_j,\psi_j\in\R^{n+1}\,,\ j\in[\omega]\,.
\end{equation}
In particular, when $\varphi_1=\cdots=\varphi_{\omega}=\varphi$ and $\psi_1=\cdots=\psi_{\omega}=\psi$, we obtain the identities
\[
\left\langle(\varphi^*)^{\omega},(\psi^*)^{\omega}\right\rangle_\mBW = \langle \varphi,\psi\rangle^{\omega}\,,\quad q_\mBW((\psi^*)^{\omega})=q(\psi)^{\omega}\,.
\]

We recall the following well-known properties of the Bombieri-Weyl inner product of two homogeneous polynomials.

\begin{lemma}\label{lem: properties BW inner product}
Let $\langle\,,\rangle$ be an inner product on $\R^{n+1}$ and $\omega$ a positive integer. Let $\langle\,,\rangle_\mBW$ be the Bombieri-Weyl inner product on $\R[x]_{\omega}$ associated with $\langle\,,\rangle$. For any two polynomials $f=(f_\alpha)_{|\alpha|=\omega}$ and $g=(g_\alpha)_{|\alpha|=\omega}$ in $\R[x]_{\omega}$, written as
\[
f(x_0,\ldots,x_n) = \sum_{|\alpha|=\omega}\binom{\omega}{\alpha}f_\alpha x^\alpha\,,\quad g(x_0,\ldots,x_n) = \sum_{|\alpha|=\omega}\binom{\omega}{\alpha}g_\alpha x^\alpha\,,
\]
we have the identities
\begin{equation}\label{eq: BW inner product standard}
\langle f,g\rangle_\mBW = \sum_{|\alpha|=\omega}\binom{\omega}{\alpha}f_\alpha g_\alpha\quad\text{and}\quad q_\mBW(f) = \sum_{|\alpha|=\omega}\binom{\omega}{\alpha}f_\alpha^2\,.
\end{equation}
\end{lemma}
\begin{proof}
Using the bilinearity of $\langle\,,\rangle_\mBW$, we obtain
\begin{align*}
\langle f,g\rangle_\mBW &= \left\langle \sum_{|\alpha|=\omega}\binom{\omega}{\alpha}f_\alpha x^\alpha,\sum_{|\beta|=\omega}\binom{\omega}{\beta}g_\beta x^\beta\right\rangle_\mBW = \sum_{|\alpha|=|\beta|=\omega}\binom{\omega}{\alpha}\binom{\omega}{\beta}f_\alpha g_\beta\langle x^\alpha,x^\beta\rangle_\mBW\,.
\end{align*}
Applying Definition \ref{def: Bombieri-Weyl inner product}, we get that $\langle x^\alpha,x^\beta\rangle_\mBW\neq 0$ if and only if $\alpha=\beta$, in which case $\langle x^\alpha,x^\alpha\rangle_\mBW=\binom{\omega}{\alpha}^{-1}$.
Then the first identity in \eqref{eq: BW inner product standard} follows. The second identity in \eqref{eq: BW inner product standard} follows immediately from the first one and the definition of $q_\mBW$.
\end{proof}

\begin{proposition}\label{prop: equivalence}
Let $\langle\,,\rangle$ be an inner product on $\R^{n+1}$ with associated quadratic form $q$, $\omega$ a positive integer, and consider the Bombieri-Weyl inner product $\langle\,,\rangle_\mBW$ on  $\R[x]_{\omega}$ associated to $\langle\,,\rangle$. Let $X^\mR\subseteq\PP_{\mR}^n$ be a real projective variety and $f\in\R[x]_{\omega}$.
Consider the projective variety $X\coloneqq\overline{X^\mR}\subseteq\PP^n$ and the image $\nu_{\omega}(X)$ under the Veronese map.
A point $\lambda(\psi^*)^{\omega}\in C(\nu_{\omega}(X))$ with $\psi\in S_q$ and $\lambda\in\C\setminus\{0\}$ belongs to $\mathrm{Crit}(\dist_{q_\mBW,f}^2,C(\nu_{\omega}(X)))$ if and only if $\psi\in\mathrm{Crit}(f,C(X)\cap S_q)$ and $\lambda = f(\psi)$.

Furthermore, let $\widetilde{\lambda}$ be maximum, in absolute value, among all normalized $X$-eigenvalues of $f$ whose corresponding normalized $X$-eigenvector $\widetilde{\psi}$ is real. Then $\widetilde{\lambda}(\widetilde{\psi}^*)^{\omega}$ is the closest point on $C(\nu_\omega(X^\mR))$ to $f$ in Bombieri-Weyl inner product.
\end{proposition}
\begin{proof}
By Definition \ref{def: critical}, a point $\lambda(\psi^*)^{\omega}\in C(\nu_{\omega}(X))$ with $\lambda\neq 0$ belongs to the subset $\mathrm{Crit}(\dist_{q_\mBW,f}^2,C(\nu_{\omega}(X)))$ if and only if $\psi$ is a nonsingular point of $C(X)\cap S_q$ and
\[
\left\langle f-\lambda(\psi^*)^{\omega},v\right\rangle_\mBW = 0\quad\forall\,v\in T_{\lambda(\psi^*)^{\omega}}C(\nu_{\omega}(X)) = T_{(\psi^*)^{\omega}}C(\nu_{\omega}(X))\,,
\]
where $T_{(\psi^*)^{\omega}}C(\nu_{\omega}(X))=\{(\psi^*)^{\omega-1}\varphi^*\mid \varphi\in T_{\psi}C(X)\}$ and the tangent spaces of $C(\nu_{\omega}(X))$ at the points $(\psi^*)^{\omega}$ and $\lambda(\psi^*)^{\omega}$ coincide because $C(\nu_{\omega}(X))$ is a cone. Equivalently, we have that
\begin{equation}\label{eq: identities tangent space}
\left\langle f-\lambda(\psi^*)^{\omega},(\psi^*)^{\omega-1}\varphi^*\right\rangle_\mBW = 0\quad\forall\,\varphi\in T_{\psi}C(X)\,.
\end{equation}
Fix an orthonormal basis $\{e_0,\ldots,e_n\}$ of $\R^{n+1}$ with respect to $\langle\,,\rangle$, with corresponding dual basis $\{x_0,\ldots,x_n\}$ of $(\R^{n+1})^*$. In particular, write $\varphi=(\varphi_0,\ldots,\varphi_n)^\mT$ in the given basis and $\varphi^*=\varphi_0x_0+\cdots+\varphi_nx_n$. On the one hand, expanding $f\in\C[x]_{\omega}$ one verifies that $\langle f,(\psi^*)^{\omega-1}x_i\rangle_\mBW=\frac{1}{\omega}\partial_i f(\psi)$ for all $i\in\{0,\ldots,n\}$, where $\partial_i f$ is the derivative of $f$ in the direction $e_i$. On the other hand, using the identities \eqref{eq: identities BW inner product} and $q(\psi)=1$, we have $\langle (\psi^*)^{\omega},(\psi^*)^{\omega-1}x_i\rangle_\mBW=\psi_i$ for all $i\in\{0,\ldots,n\}$\,. Plugging these identities in \eqref{eq: identities tangent space} yields
\begin{equation}\label{eq: simplified identities tangent space}
\left\langle\frac{1}{\omega}\nabla f(\psi),\varphi\right\rangle = \lambda\langle\psi,\varphi\rangle \quad\forall\,\varphi\in T_{\psi}C(X)\,,
\end{equation}
namely a linear combination of the vectors $\nabla f(\psi)=(\partial_0 f(\psi),\ldots,\partial_n f(\psi))$ and $2\,\psi=\nabla(q(\psi)-1)$ is in the normal space $N_{\psi}C(X)$ with respect to $\langle\,,\rangle$. Assuming that $X\subseteq V$ is defined by the radical ideal generated by the subset $F=\{f_1,\ldots,f_s\}\subseteq\C[x]$, we conclude that a linear combination of $\nabla f(\psi)$ and $\nabla(q(\psi)-1)$ is in the row span of the Jacobian matrix $J_{\psi}(F)=\left(\partial_j f_i\right)_{i,j}$. Therefore, $\psi$ is critical for the polynomial objective function $f\colon S_q\to\C$ constrained to $C(X)$. Note that $\psi\in T_{\psi}C(X)$ because $C(X)$ is a cone, so that we can evaluate \eqref{eq: simplified identities tangent space} at $\varphi=\psi$. Then Euler's identity of homogeneous polynomials yields
\[
f(\psi) = \left\langle\frac{1}{\omega}\nabla f(\psi),\psi\right\rangle = \lambda\langle\psi,\psi\rangle = \lambda\,.
\]
The previous implications can be reversed, thus showing the first part of the statement.
The last part of the statement descends by a similar argument used in the proof of \cite[Theorem 2.19]{qi2017spectral}, see also \cite[Remark 5.2]{dirocco2025osculating}.
\end{proof}

\begin{corollary}\label{corol: RR degree equals BW degree}
Consider a real projective variety $X^\mR\subseteq\PP_{\mR}^n$ and let $X=\overline{X^\mR}\subseteq\PP^n$. Let $q$ be a positive-definite quadratic form on $\R^{n+1}$, $\omega$ a positive integer, and $q_\mBW$ the Bombieri-Weyl quadratic form associated to $q$ and $\omega$. Denote by $Q_\mBW$ the quadric hypersurface in $\PP(\C[x]_{\omega})$ cut out by $q_\mBW$. Then
\[
\RRdeg_\omega(X) = \mathrm{DD}(\nu_{\omega}(X),Q_\mBW)\,.
\]
\end{corollary}
\begin{proof}
By definition, the distance degree $\mathrm{DD}(\nu_{\omega}(X),Q_\mBW)$ equals the cardinality of the finite set $\mathrm{Crit}(\dist_{q_\mBW,f}^2,C(\nu_{\omega}(X)))$ for a generic polynomial function $f\colon\C^{n+1}\to\C$ of degree $\omega$. By Proposition~\ref{prop: equivalence}, the set $\mathrm{Crit}(\dist_{q_\mBW,f}^2,C(\nu_{\omega}(X)))$ is in bijection with $\mathrm{Crit}(f,C(X)\cap S_q)$. This shows that there exists a Zariski open subset of polynomials $f$ such that $\{[\psi]\in\PP^n \mid\text{$[\psi]$ is an $X$-eigenpoint of $f$}\}$ consists of finitely many points, and its cardinality is the maximum one. The latter is by definition $\RRdeg_\omega(X)$.
\end{proof}

\begin{example}
We study the normalized eigenpairs of the binary cubic
\[
f(x_0,x_1)=2\,x_0^3-3\,x_0^2x_1+6\,x_0x_1^2-x_1^3\,,
\]
in particular, $X=\PP^1$ in this case. Following Remark \ref{rmk: normalized eigenvectors}, we compute the eigenpoints of $f$, namely the fixed points of $\nabla f\colon\PP^1\dashrightarrow\PP^1$. Hence these are the points $[\psi_0:\psi_1]\in\PP^1$ such that the rank of
\[
A_f(\psi)\coloneqq
\begin{pmatrix}
    \nabla f(\psi) \\
    \psi
\end{pmatrix}
=
\begin{pmatrix}
    6(\psi_0^2-\psi_0\psi_1+\psi_1^2) & -3(\psi_0^2-4\,\psi_0\psi_1+\psi_1^2)\\
    \psi_0 & \psi_1
\end{pmatrix}
\]
is less than two, or equivalently $0 = \det A_f(\psi) = (\psi_0-2\,\psi_1)(\psi_0-\psi_1)(\psi_0+\psi_1)$. This gives the locus $\{[2:1],[1:1],[1:-1]\}\subseteq\PP^1$ of eigenpoints of $f$, corresponding to the normalized eigenpairs in the following table:
\begin{center}
\begin{tabular}{c||c|c|c}
		$\psi$ & $\pm(\frac{2\sqrt{5}}{5},\frac{\sqrt{5}}{5})$ & $\pm(\frac{\sqrt{2}}{2},\frac{\sqrt{2}}{2})$ & $\pm(\frac{\sqrt{2}}{2},-\frac{\sqrt{2}}{2})$ \\[2pt]\hline
		$\lambda=f(\psi)$ & $\pm\frac{3\sqrt{5}}{5}$ & $\pm\sqrt{2}$ & $\pm 3\sqrt{2}$
\end{tabular}
\end{center}
The three eigenpairs above correspond to the polynomials $g_1=\frac{3}{25}(2x_0+x_1)^3$, $g_2=\frac{1}{2}(x_0+x_1)^3$, and $g_3=\frac{3}{2}(x_0-x_1)^3$ on the affine cone over the Veronese variety $V_3\subseteq\K[x]_3$, respectively. The points and the corresponding eigenpairs are colored in this order in orange, blue, and magenta in Figure \ref{fig: eigenvectors binary cubic}. Considering the notations of Lemma \ref{lem: properties BW inner product}, the coordinates of $f$ are $(2,-1,2,-1)$, while $g_1=(\frac{24}{25},\frac{12}{25},\frac{6}{25},\frac{3}{25})$, $g_2=(\frac{1}{2},\frac{1}{2},\frac{1}{2},\frac{1}{2})$, and $g_3=(\frac{3}{2},-\frac{3}{2},\frac{3}{2},-\frac{3}{2})$. Using the formulas in \eqref{eq: BW inner product standard}, one verifies that the squared distances $\dist_{q_\mBW,f}^2(g_i)=\langle f-g_i,f-g_i\rangle_\mBW$ between $f$ and the polynomials $g_i$ are respectively $18.2$, $18$, and $2$. In particular, $g_3$ is the closest rank-one binary form of degree $3$ to $f$ in Bombieri-Weyl norm, and corresponds to the highest normalized eigenvalue of $f$.\hfill$\diamondsuit$

\begin{figure}
\begin{minipage}{.47\textwidth}
\centering
\begin{overpic}[width=0.7\textwidth]{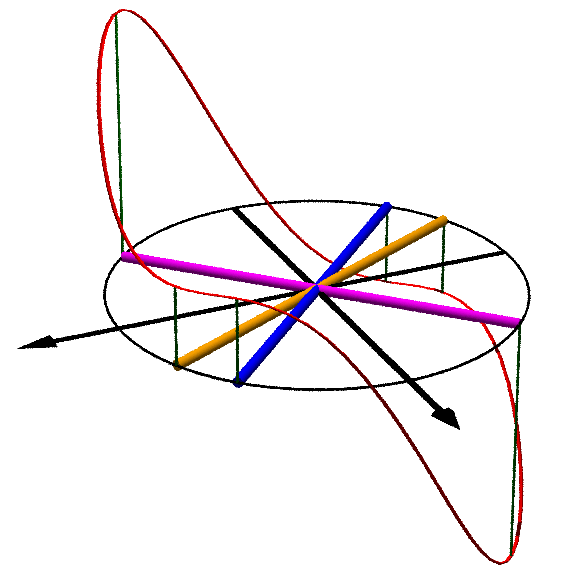}
	\put (3,34) {{\large $v_0$}}
	\put (75,20) {{\large $v_1$}}
\end{overpic}
\end{minipage}
\begin{minipage}{.47\textwidth}
\hspace*{-0.7cm}
\centering
\begin{overpic}[width=\textwidth]{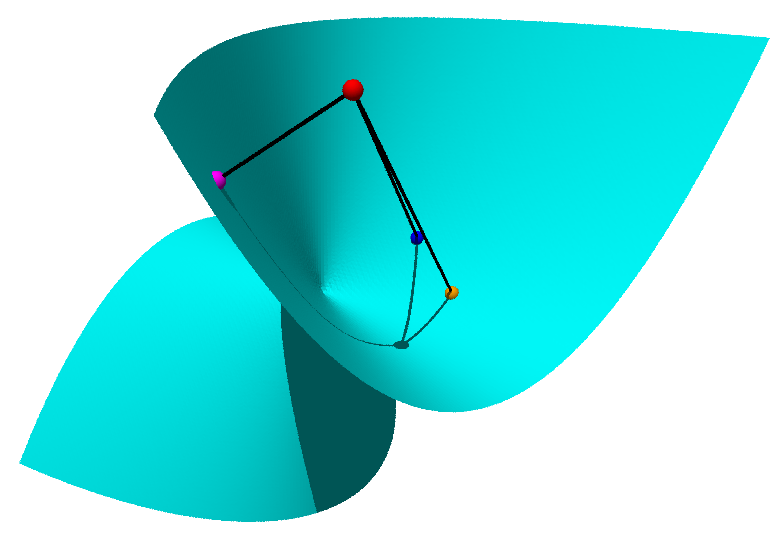}
	\put (50,58) {{\large \red{$f$}}}
	\put (75,57) {{\large $C(V_3^\mR)$}}
\end{overpic}
\end{minipage}
\caption{Comparison between normalized eigenvectors of $f$ (on the left) and critical points of the Bombieri-Weyl distance function from $f$ (on the right).}\label{fig: eigenvectors binary cubic}
\end{figure}
\end{example}

In Corollary \ref{corol: RR degree equals BW degree}, we have established an identity between the Rayleigh-Ritz degree of a projective variety $X$ and a particular distance degree of a Veronese embedding of $X$. Building on this, the following sections explore concrete formulas that compute this invariant, specifically for nonsingular varieties defined either implicitly or as images of polynomial maps. Later on, we focus on varieties of rank-one tensors.

\section{Porteous-type formulas for Rayleigh-Ritz degrees}\label{sec: Porteous}

We begin by providing an intersection theory toolkit needed for the upcoming results. We refer to \cite{fulton1998intersection} for further details.

Consider an irreducible projective variety $X\subseteq\PP^n$ of dimension $m$. We denote by $A^i(X)$ the $i$-th Chow group of $X$, encoding formal linear combinations with integer coefficients of $i$-codimensional irreducible subvarieties of $X$, up to rational equivalence. The Chow ring of $X$ is $A^*(X)=\bigoplus_{i\ge 0}A^i(X)$, equipped with the intersection product.

Let $\ke\to X$ be a vector bundle on $X$ of rank $r$.
For every $i\in\{0,\ldots,r\}$, one can associate to $\ke$ a rational equivalence class $c_i(\ke)\in A^i(X)$, called the {\em $i$-th Chern class} of $\ke$. The {\em total Chern class} of $\ke$ is $c(\ke)\coloneqq\sum_{i=0}^rc_i(\ke)$.
A first classical example is the line bundle $\ko_X(1)\coloneqq \iota^*(\ko_{\PP^n}(1))$ associated to an embedding $\iota\colon X\hookrightarrow\PP^n$. In the following, we simply write $\kl=\ko_X(1)$.
Another fundamental vector bundle on $X$ is the tangent bundle $\kt_X$ of rank $m$. The Chern classes of $X$ are by definition the Chern classes of $\kt_X$, and we denote them simply by $c_i(X)$. For the sake of brevity, we also use the shorthand $c_i$ to indicate $c_i(X)$ when $X$ is clear from the context.
Let $L=c_1(\kl)$. 
For every $i\in\{0,\ldots,m\}$, the class $c_i(\ke)\cdot L^{m-i}$ is equal to an integer times the class of a point in $\PP^n$. That integer is the degree of $c_i(\ke)$, denoted by $\int_X c_i(\ke)\cdot L^{m-i}$. In particular, the degree of $c_m(\ke)\in A^m(X)$ is called the {\em top Chern number of $\ke$}. Furthermore, given a short exact sequence $0\to\ke\to\kf\to\kg\to 0$ of vector bundles on $X$, the total Chern classes satisfy the {\em Whitney sum} property $c(\kf)=c(\ke)\cdot c(\kg)$.

Let $\iota\colon X\hookrightarrow Y$ be a closed embedding of nonsingular varieties with normal bundle $\kn_{X/Y}$. From the exact sequence
\begin{equation}\label{eq: seq adjunction}
    0\to \kt_X \longrightarrow \kt_Y|_X \longrightarrow \kn_{X/Y} \to 0
\end{equation}
and the Whitney sum property, we derive the identity $c(X) = c(\kt_Y|_X)/c(\kn_{X/Y})$.
If $X=Y=\PP^n$, then $c(\PP^n)=(1+h)^{n+1}$, where $h=c_1(\ko_{\PP^n}(1))$ is the class of a hyperplane in $\PP^n$. In general we have $c(\kt_{\PP^n}|_X)=(1+L)^{n+1}$ and $c(X) = (1+L)^n/c(\kn_{X/\PP^n})$.
Hence, knowing $c(\kn_{X/\PP^n})$ allows us to compute the Chern classes of the embedding $\iota$. An important example is when $X\hookrightarrow\PP^n$ is the complete intersection of $c$ hypersurfaces in $\PP^n$ of degrees $\delta_1,\ldots,\delta_c$. In this case $\kn_{X/\PP^n} = \bigoplus_{i=1}^c\ko_X(\delta_i)$, hence by the Whitney sum property
\[
c(X) = \frac{c(\kt_{\PP^n}|_X)}{c(\kn_{X/\PP^n})} =  \frac{(1+L)^{n+1}}{\prod_{i=1}^c c(\ko_X(\delta_i))} = \frac{(1+L)^{n+1}}{\prod_{i=1}^c(1+\delta_iL)}\,.
\]
For example, if $X\hookrightarrow\PP^n$ is a hypersurface of degree $\delta$, then
\begin{equation}\label{eq: total Chern class hypersurface}
c(X) = \frac{(1+L)^{n+1}}{1+\delta L} = \sum_{i=0}^{n-1}\left(\sum_{j=0}^i\binom{n+1}{j}(-\delta)^{i-j}\right)L^i\,.
\end{equation}
Another important example is the Veronese variety $X=V_d$, image of the Veronese embedding $\nu_d\colon \PP^m\hookrightarrow \PP^M$, where $M=\binom{m+d}{d}-1$. Then $\nu_d^*\ko_{\PP^M}(1)=\ko_{\PP^m}(d)$, in particular $L=dh$ with $h=c_1(\ko_{\PP^m}(1))$ and
\begin{equation}\label{eq: total Chern class normal bundle Veronese}
c(\kn_{\PP^m/\PP^M}) = \frac{c(\kt_{\PP^M}|_{\PP^m})}{c(\PP^m)} = \frac{(1+dh)^{M+1}}{(1+h)^{m+1}}\,.
\end{equation}
We discuss Chern classes of Segre-Veronese embeddings in Proposition \ref{prop: degrees Chern classes Segre-Veronese}.

In the following theorem, we determine an upper bound for the RR degree of a variety defined implicitly, which is an equality for generic complete intersection varieties.

\begin{theorem}\label{thm: RR degree complete intersection}
Let $X\subseteq\PP^n$ be a variety of codimension $c$ cut out by the polynomials $f_1,\dots,f_m$ of degrees $\delta_1,\dots,\delta_m$ such that the first $c$ of them form a regular sequence. Then for all $\omega\ge 1$
\begin{equation}\label{eq: RR degree complete intersection}
    \RRdeg_\omega(X)\leq  \delta_1\cdots\delta_c \sum_{\substack{i_0+\cdots+i_c=n-c\\i_0,\dots,i_c\geq 0}}\sum_{\ell=0}^{i_0}(\omega-1)^\ell\cdot\prod_{k=1}^c(\delta_k-1)^{i_k}\,.
\end{equation}
The equality holds if $X$ is a generic complete intersection.
\end{theorem}
\begin{proof}
By assumption, the variety $X'$ cut out by the polynomials $f_1,\dots,f_c$ is a complete intersection. Then, $X$ is an irreducible component of $X'$ and so, $\RRdeg_{\omega}(X)\leq\RRdeg_{\omega}(X')$. Since the $\RRdeg_{\omega}(X')$ is bounded from above by the Rayleigh-Ritz degree of a generic complete intersection with parameters $d_1,\dots, d_c$, to prove the theorem, it suffices to study the case of generic complete intersections.

Consider a homogeneous polynomial $f$ of degree $\omega$, and let $f_1(x),\ldots,f_c(x)$ be the homogeneous polynomials of degrees $\delta_1,\ldots,\delta_c$ that define the generic complete intersection $X$. Without loss of generality, we assume that $q(x)=x_0^2+\cdots+x_n^2$.

A point $[\psi]\in\PP^n$ is an $X$-eigenpoint of $f$ if and only if $\psi\in\mathrm{Crit}(f,C(X)\cap S_q)$, namely $\psi$ is a nonsingular point of $C(X)\cap S_q$ and the gradient of $f$ at $\psi$ is in the normal space of $C(X)\cap S_q$ at $\psi$. The latter normal space is the linear space generated by the gradients of the polynomials $f_1,\ldots,f_c$ and $q$. This is equivalent to requiring that
\begin{equation}\label{eq: matrix rank c + 1}
\rank
\begin{pmatrix}
    \psi_0 & \cdots & \psi_n \\[2pt]
    \frac{\partial f}{\partial x_0}(\psi) & \cdots & \frac{\partial f}{\partial x_n}(\psi) \\[2pt]
    \frac{\partial f_1}{\partial x_0}(\psi) & \cdots & \frac{\partial f_1}{\partial x_n}(\psi) \\[2pt]
    \vdots&\ddots&\vdots \\[2pt]
    \frac{\partial f_c}{\partial x_0}(\psi) & \cdots & \frac{\partial f_c}{\partial x_n}(\psi)
\end{pmatrix}
\le c+1\,.
\end{equation}
Since $f_1,\ldots,f_c$ are generic, the variety $C(X)\cap S_q$ is nonsingular of codimension $c+1$ (see for example \cite[II, Exercise 8.4(d)]{hartshorne1977algebraic}). This means that, for every $\psi\in C(X)\cap S_q$, the rank of the above matrix is at least $c+1$. Furthermore, applying Corollary \ref{corol: RR degree equals BW degree}, there exists a Zariski open subset of polynomials $f$ with the following property: the locus of $\psi\in C(X)\cap S_q$ such that the condition \eqref{eq: matrix rank c + 1} is satisfied consists of finitely many points, and its cardinality is the maximum one, equal by definition to $\RRdeg_\omega(X)$.\\
The rows of the $(c+2)\times(n+1)$ matrix above are vectors of homogeneous polynomials of degrees $1$, $\omega-1$, $\delta_1-1$,\ldots, $\delta_c-1$, respectively. Equivalently, they define elements of the cohomology groups $H^0(\PP^n,\ko_{\PP^n}(1))$, $H^0(\PP^n,\ko_{\PP^n}(\omega-1))$, $H^0(\PP^n,\ko_{\PP^n}(\delta_1-1))$,\ldots, $H^0(\PP^n,\ko_{\PP^n}(\delta_c-1))$. After pulling back via the embedding $X\hookrightarrow\PP^n$, we can say that the matrix in \eqref{eq: matrix rank c + 1} represents the fiber $\varphi_{[\psi]}$ of a vector bundle morphism $\varphi\colon \ke\to \kf$, where
\[
\ke\coloneqq\ko_X^{\oplus(n+1)}\,,\quad \kf\coloneqq \ko_X(1)\oplus\ko_X(\omega-1)\oplus\bigoplus_{i=1}^c\ko_X(\delta_i-1)\,.
\]
Therefore, for a generic $f$ the locus of $X$-eigenpoints $[\psi]$ of $f$ coincides with the degeneracy locus $D_{c+1}(\varphi)\coloneqq\{[\psi]\in X\mid \rank\varphi_{[\psi]}\le c+1\}$. We have previously shown that, for a generic $f$, the dimension of $D_{c+1}(\varphi)$ is zero. This agrees with the expected dimension of $D_{c+1}(\varphi)$, that is $\dim X -[(c+2)-(c+1)]\cdot[n+1-(c+1)]=n-c-(n-c)=0$. Hence, we conclude that
\begin{equation}\label{eq: identity RR degree complete intersection}
\RRdeg_\omega(X) = \int_X [D_{c+1}(\varphi)] = (-1)^{n-c}\int_X c_{n-c}(\ke-\kf)\,,
\end{equation}
where in the last identity we applied Porteous' formula \cite[II, \S4]{arbarello1985geometry}.
The total Chern class $c(\ke-\kf)$ is equal to
\begin{align*}
   c(\ke-\kf) &= \frac{c(\ke)}{c(\kf)} = \frac{1}{c(\kf)} = \frac{1}{(1+L)(1+(\omega-1)L)\prod_{k=1}^c(1+(\delta_k-1)L)}\\
   &= \left(\sum_{j=0}^{\infty}(-1)^jL^j\right)\cdot\left(\sum_{\ell=0}^{\infty}(-1)^\ell(\omega-1)^\ell L^\ell\right)\cdot\prod_{k=1}^c\left(\sum_{i_k=0}^{\infty}(-1)^{i_k}(\delta_k-1)^{i_k}L^{i_k}\right)\\
   &= \sum_{i_0=0}^{\infty}(-1)^{i_0}\sum_{\ell=0}^{i_0}(\omega-1)^\ell L^{i_0} \cdot\prod_{k=1}^c\left(\sum_{i_k=0}^{\infty}(-1)^{i_k}(\delta_k-1)^{i_k}L^{i_k}\right)\,.
\end{align*}
Therefore, we obtain the identity
\[
    c_{n-c}(\ke-\kf) =(-1)^{n-c}\left[\sum_{\substack{i_0+\cdots+i_c=n-c\\i_0,\dots,i_c\geq 0}}\sum_{\ell=0}^{i_0}(\omega-1)^\ell\cdot\prod_{k=1}^c(\delta_k-1)^{i_k}\right]L^{n-c}\,.
\]
Plugging the previous identity in \eqref{eq: identity RR degree complete intersection} and knowing that $\int_X L^{n-c}=\deg(X)=\delta_1\cdots\delta_c$, we obtain the desired degree in \eqref{eq: RR degree complete intersection}.
\end{proof}

\begin{remark}
Motivated by the constrained Rayleigh quotient minimization problem explained in the Introduction, we simplify the inequality \eqref{eq: RR degree complete intersection} for $\omega=2$ and $\delta_1=\cdots=\delta_m=\delta\in\{1,2\}$: 
\begin{enumerate}
    \item if $\delta=1$, then $X$ is a projective subspace of codimension $c$, hence a complete intersection. The product $\prod_{k=1}^c(\delta_k-1)^{i_k}$ is nonzero only if $i_1=\cdots=i_k=0$, which forces $i_0=n-c$, hence $\RRdeg_\omega(X)\leq n-c+1$.
    \item if $\delta=2$, then $X$ is a variety of codimension $c$ cut out by quadratic polynomials. In this case the product $\prod_{k=1}^c(\delta_k-1)^{i_k}$ is always equal to one and
    \[
        \RRdeg_\omega(X)\leq  2^c\sum_{\substack{i_0+\cdots+i_c=n-c\\i_0,\dots,i_c\geq 0}}(i_0+1) = 2^c\sum_{i_0=0}^{n-c}(i_0+1)\binom{n-i_0-1}{c-1} = 2^c\binom{n+1}{c+1}\,.
    \]
\end{enumerate}
\end{remark}

\begin{remark}
Let $X\subseteq\PP^n$ be a variety of codimension $c$ cut out by the polynomials $f_1,\dots,f_m$ of degrees $\delta_1\geq\dots\geq\delta_m$. Then, the previous result implies that
\begin{equation*}
    \RRdeg_\omega(X)\leq  \delta_1\cdots\delta_c \sum_{\substack{i_0+\cdots+i_c=n-c\\i_0,\dots,i_c\geq 0}}\sum_{\ell=0}^{i_0}(\omega-1)^\ell\cdot\prod_{k=1}^c(\delta_k-1)^{i_k}\,,
\end{equation*}
for all $\omega\ge 1$. Clearly, this inequality is weaker than Equation~\eqref{eq: RR degree complete intersection}, but it can be applied more easily since it does not require finding a regular sequence of length $c$ among the generators of $X$.    
\end{remark}

It is interesting to compare the value of $\RRdeg_\omega(X)$ in Theorem \ref{thm: RR degree complete intersection} and the generic distance degree $\mathrm{gDD}(X)$, which is studied in \cite[Proposition 2.6]{draisma2016euclidean}.

\begin{corollary}\label{corol: BW degree hypersurface}
Let $X\subseteq\PP^n$ be a generic hypersurface of degree $\delta$. Then for all $\omega\ge 1$
\[
    \RRdeg_\omega(X) = \delta\sum_{i=0}^{n-1}\sum_{\ell=0}^i(\omega-1)^\ell\cdot(\delta-1)^{n-1-i}\,.
\]
\end{corollary}
\begin{example}
Consider the plane curve $X\subset\PP^2$ cut out by the Fermat cubic $f_1=x_0^3+x_1^3+x_2^3$, and let $q(x)=x_0^2+x_1^2+x_2^2$ be the standard Euclidean quadratic form.
By Corollary~\ref{corol: BW degree hypersurface}, we expect to have $\RRdeg_2(X)=12$. In fact, given a generic $f\in\C[x_0,x_1,z_2]_2$, the $X$-eigenpoints of $f$ are the points in $X$ that are also solutions of
\[
    h(x_0,x_1,x_2) \coloneqq \det
    \begin{pmatrix}
        x_0 & x_1 & x_2 \\[2pt]
        \frac{\partial f}{\partial x_0} & \frac{\partial f}{\partial x_1} & \frac{\partial f}{\partial x_2} \\[2pt]
        x_0^2 & x_1^2 & x_2^2
    \end{pmatrix}
    =0\,.
\]
Therefore, if the curves defined respectively by $h$ and $f$ are in general position, by B\'ezout's theorem the number of $X$-eigenpoints of $f$ is $3\cdot 4=12$. For instance, let $f=x_0^2+10\,x_1^2+2\,x_1x_2+8\,x_2^2$. The $X$-eigenpoints of $f$ are the solutions in $\PP^2$ of the system
\[
    \begin{cases}
        f_1 = x_0^3+x_1^3+x_2^3=0\\
        h = x_0(x_0x_1^2-x_1^3-2\,x_0x_1x_2-7\,x_1^2x_2-x_0x_2^2+9\,x_1x_2^2+x_2^3)=0\,.
    \end{cases}
\]
For $x_0=0$, we find three solutions $[0:1:-\xi],[0:1:-\xi^2],[0:1:-1]$ where $\xi$ is a cubic root of $1$. For $x_0=1$, we find nine distinct solutions, of which three are real and six are complex. So, in total, we have twelve $X$-eigenpoints as predicted.

For each of the four real $X$-eigenpoints of $f$, we compute the associated $X$-eigenvalue by evaluating the quadratic form $f$ at the corresponding affine point on the sphere. Among these points, the largest $X$-eigenvalue is associated with the normalized $X$-eigenvector $\widetilde{\psi}=(0,\sqrt{2}/2,-\sqrt{2}/2)$ and is equal to $\widetilde{\lambda}=f(\widetilde{\psi})=8$.
By Proposition~\ref{prop: equivalence}, the closest point to $f$ on $C(\nu_2(X^\mR))$ with respect to the Bombieri-Weyl inner product is $\widetilde{\lambda}(\widetilde{\psi}^*)^2=4(x_1-x_2)^2$. Using \eqref{eq: BW inner product standard}, one verifies that $\dist_{q_\mBW,f}^2(\widetilde{\lambda}(\widetilde{\psi}^*)^2)=103$.\hfill$\diamondsuit$
\end{example}

Several varieties are represented as images of polynomial maps rather than implicitly. One can directly use the parametrization of $X$ to study its $X$-eigenpoints.

\begin{lemma}\label{lem: polynomial map}
Consider an injective polynomial map
\[
g\colon\PP^m\dashrightarrow\PP^n\,,\quad[t]\mapsto[g_0(t):\cdots:g_n(t)]
\]
for some homogeneous polynomials $g_0,\ldots,g_n$ of the same degree. Let $B$ be the base locus of $g$ and define $X\coloneqq\overline{g(\PP^m\setminus B)}\subseteq\PP^n$. Let $q$ be the quadratic form associated with the fixed isotropic quadric. Consider a generic polynomial function $f\colon\C^{n+1}\to\C$ and the $2\times(m+1)$ matrix
\begin{equation}\label{eq: matrix A}
A(w) = 
\begin{pmatrix}
    \nabla_t(q\circ g)(w)\\
    \nabla_t(f\circ g)(w)
\end{pmatrix}
=
\begin{pmatrix}
    \frac{\partial(q\circ g)}{\partial t_0}(w) & \cdots & \frac{\partial(q\circ g)}{\partial t_m}(w) \\[2pt]
    \frac{\partial(f\circ g)}{\partial t_0}(w) & \cdots & \frac{\partial(f\circ g)}{\partial t_m}(w)
\end{pmatrix}\,,\quad w\in\C^{m+1}\,.
\end{equation}
Then $g$ defines a bijection between $\{[w]\in\PP^m\setminus B\mid\text{$\rank Jg(w)=m+1$ and $\rank A(w)\le 1$}\}$ and $\{[\psi]\in\PP^n \mid\text{$[\psi]$ is an $X$-eigenpoint of $f$}\}$, where $Jg$ is the Jacobian matrix of $g$.
\end{lemma}
\begin{proof}
Let $[w]=[w_0:\cdots:w_m]\in\PP^m$ such that $\rank A(w)\le 1$, namely there exist $[\lambda,\mu]\in\PP^1$ such that $\lambda_t\nabla(q\circ g)(w)+\mu\nabla_t(f\circ g)(w)=0$. Applying the chain rule, this is equivalent to having $\lambda\nabla_xq(g(w))+\mu\nabla_xf(g(w))$ in the left kernel of $Jg(w)$. Since $\rank Jg(w)=m+1$ by assumption, this implies that $\psi=g(w)\in\mathrm{Crit}(f,S_q\cap C(X))$, that is, $[\psi]$ is an $X$-eigenvector of $f$. The injectivity of the map $g$ gives the desired bijective correspondence.
\end{proof}

Using the previous lemma and Porteous' formula, we compute the RR degree of the image of a generic projective morphism. Note that the case $d=1$ agrees with \cite[Theorem 5.6]{cartwright2013number}, in particular the RR degree equals the number of eigenpoints of a generic homogeneous polynomial $f$ of degree $\omega$, see Remark~\ref{rmk: normalized eigenvectors}.

\begin{proposition}\label{prop: RR degree generic morphism}
Assume $n>m$. Let $g_0,\ldots,g_n$ be $n+1$ generic homogeneous polynomials of degree $d$ in $m+1$ variables, and assume that the image $X$ of the morphism $g\colon\PP^m\to\PP^n$ associated with the polynomials $g_i$ is nonsingular. Then
\begin{equation}
    \RRdeg_\omega(X) =
    \begin{cases}
    (m+1)(2d-1)^m & \text{if $\omega=2$}\\
    \frac{(\omega d-1)^{m+1}-(2d-1)^{m+1}}{d(\omega-2)} & \text{if $\omega>2$.}
    \end{cases}
\end{equation}
\end{proposition}
\begin{proof}
The genericity of the polynomials $g_i$ ensures that $g$ is defined everywhere, that is, its base locus is empty.
Similarly as in the proof of Theorem~\ref{thm: RR degree complete intersection}, by Corollary \ref{corol: RR degree equals BW degree} and Lemma \ref{lem: polynomial map} there exists a Zariski open subset of polynomials $f$ with the following property: $f$ admits finitely many $X$-eigenpoints, and their number coincides with the cardinality of the set $\{[w]\in\PP^m\mid\rank A(w)\le 1\}$, where $A(w)$ is defined in \eqref{eq: matrix A}.
The two rows of $A(w)$ have homogeneous polynomial entries of degrees $2d-1$ and $\omega d-1$, respectively. Therefore $A(w)$ defines one fiber of the vector bundle morphism $\varphi\colon\ke\to\kf$, where $\ke\coloneqq\ko_{\PP^m}^{\oplus(m+1)}$ and $\kf\coloneqq \ko_{\PP^m}(2d-1)\oplus\ko_{\PP^m}(\omega d-1)$. In particular, the set of $X$-eigenpoints of a generic $f$ is in bijection with the degeneracy locus $D_1(\varphi)$, in particular $D_1(\varphi)$ is zero-dimensional. This agrees with the expected dimension of $D_1(\varphi)$, that is $m-(2-1)m=m-m=0$. We conclude that
\begin{equation}\label{eq: identity RR degree generic morphism}
\RRdeg_\omega(X) = \int_{\PP^m} [D_1(\varphi)] = (-1)^m \int_{\PP^m} c_m(\ke-\kf)\,,
\end{equation}
where in the last equality we applied Porteous' formula.
Denoting $h=c_1(\ko_{\PP^m}(1))$, the total Chern class of $\ke-\kf$ is
\[
c(\ke-\kf) = \frac{1}{(1+(2d-1)h)(1+(\omega d-1)h)} = \sum_{i,j=0}^{\infty}(-1)^{i+j}(2d-1)^i(\omega d-1)^jh^{i+j}\,,
\]
therefore $c_m(\ke-\kf)=(-1)^m h^m\sum_{i=0}^m(2d-1)^{m-i}(\omega d-1)^i$. Plugging this identity in \eqref{eq: identity RR degree generic morphism} and using $\int_{\PP^m}h^m = 1$, we obtain the formula in the statement.
\end{proof}

\begin{remark}
Suppose that $n=\binom{m+d}{d}-1$, and consider $n+1$ generic polynomials $g_0,\ldots,g_n$ in $\C[t]_d$. Let $X$ be the image of the morphism $g\colon\PP^m\to\PP^n$. In particular, $X$ is equal to the image of the Veronese embedding $\nu_d\colon\PP^m\hookrightarrow\PP^n$ after a linear transformation. In this case,
it is interesting to compare the degree in Proposition \ref{prop: RR degree generic morphism} with the generic distance degree $\mathrm{gDD}(\nu_\omega(X))$, computed applying \cite[Proposition 7.10]{draisma2016euclidean}:
\[
\mathrm{gDD}(\nu_\omega(X)) = \mathrm{gDD}(\nu_{\omega\nu}(\PP^m)) = \frac{(2\omega d-1)^{m+1}-(\omega d-1)^{m+1}}{\omega d}\,.
\]
\end{remark}

\section{Rayleigh-Ritz degrees of varieties in general position}\label{sec: RR degrees general position}

We now discuss a general formula for the RR degree of a variety $X\subseteq\PP^n$ in general position with respect to the isotropic quadric $Q\subseteq\PP^n$. If we denote by $\widetilde{X}$ any linear transformation of $X$, not necessarily generic, we always have the chain of inequalities
\[
\RRdeg_\omega(\widetilde{X}) \le \RRdeg_\omega(X) \le \mathrm{gDD}(\nu_\omega(X))\,.
\]
This is particularly interesting in the case of varieties of rank-one tensors, which we discuss next. The invariant $\mathrm{gDD}(\nu_{\omega}(X))$ is called {\em generic distance degree} of $\nu_\omega(X)$, and corresponds to the choice of an isotropic quadric in $\PP(\C[x]_\omega)$ which intersects $\nu_\omega(X)$ transversally, see \cite{draisma2016euclidean}. The latter satisfies the following formula.

\begin{theorem}\label{thm: RR degree X nonsingular transversal to Q}
Let $X\subseteq\PP^n$ be a nonsingular variety of dimension $m$, in general position with respect to the fixed isotropic quadric $Q\subseteq\PP^n$. The RR degree of $X$ is
\begin{equation}\label{eq: RR degree X nonsingular transversal to Q}
    \RRdeg_\omega(X) = \mathrm{gDD}(\nu_\omega(X))-(\omega-1)\mathrm{gDD}(\nu_\omega(X\cap Q))\,.
\end{equation}
Furthermore, if the embedding of $X$ in $\PP^n$ is given by a line bundle $\kl$ with $L=c_1(\kl)$, then
\begin{align}\label{eq: RR degree X nonsingular given line bundle}
\begin{split}
\RRdeg_\omega(X) &= \sum_{i=0}^m(-1)^i\left(\sum_{j=0}^{m-i}\omega^{m-i-j}2^j\right)\int_X c_i\cdot L^{m-i}\\
&=
\begin{cases}
\sum_{i=0}^m(-1)^i(m+1-i)2^{m-i}\int_X c_i\cdot L^{m-i} & \text{if $\omega = 2$}\\[2pt]
\sum_{i=0}^m(-1)^i\frac{\omega^{m+1-i}-2^{m+1-i}}{\omega-2}\int_X c_i\cdot L^{m-i} & \text{if $\omega \neq 2$.}
\end{cases}
\end{split}
\end{align}
\end{theorem}
\begin{proof}
We know that $\RRdeg_\omega(X)=\mathrm{DD}(\nu_\omega(X),Q_\mBW)$. Since $X$ is in general position with respect to $Q$, and since $Q_\mBW$ is the Bombieri-Weyl isotropic quadric in $\PP(\C[x]_\omega)$ associated to $Q$, then the intersection $\nu_\omega(X)\cap Q_\mBW$ is nonreduced and scheme-theoretically is equal to $\omega\,\nu_\omega(X\cap Q)$, in particular $\nu_\omega(X)\cap Q_\mBW$ is equisingular. The identity \eqref{eq: RR degree X nonsingular transversal to Q} is then a consequence of \cite[Corollary 3.6]{maxim2020defect}.

For the second identity, we apply \cite[Theorem 5.8]{draisma2016euclidean} twice: given a nonsingular projective variety $Y$ of dimension $m$, embedded in projective space via the line bundle $\kl$, if $L=c_1(\kl)$, then
\begin{equation}\label{eq: gen ED degree Chern}
    \mathrm{gDD}(Y) = \sum_{i=0}^{m}(-1)^i(2^{m+1-i}-1)\int_Y c_i(Y)\cdot L^{m-i}\,.
\end{equation}
Consider the embedding $\tau\colon X\cap Q\hookrightarrow X$ and let $\widetilde{L}=c_1(\ko_{X\cap Q}(1))=c_1(\tau^*\ko_X(1))$. The equation appearing three lines after \cite[Lemma 1.2]{parusinski1995chern} reads
\begin{equation}\label{eq: Chern classes intersection with hypersurface}
    \int_{X\cap Q} c_i(X\cap Q)\cdot \widetilde{L}^{m-1-i} = 2\sum_{j=0}^i(-1)^{i-j}2^{i-j}\int_X c_j \cdot L^{m-j}
\end{equation}
and allows us to write the Chern classes of $X\cap Q$ in terms of the Chern classes of $X$. Using \eqref{eq: RR degree X nonsingular transversal to Q}, \eqref{eq: gen ED degree Chern}, and \eqref{eq: Chern classes intersection with hypersurface}, we have
\[
\resizebox{\textwidth}{!}{
$\begin{aligned}
&\RRdeg_\omega(X) = \mathrm{gDD}(\nu_\omega(X))-(\omega-1)\mathrm{gDD}(\nu_\omega(X\cap Q))\\
&=\sum_{i=0}^m(-1)^i(2^{m+1-i}-1)\int_X c_i \cdot (\omega L)^{m-i}-(\omega-1)\sum_{i=0}^{m-1}(-1)^i(2^{m-i}-1)\int_{X\cap Q}c_i(X\cap Q)\cdot (\omega\widetilde{L})^{m-1-i}\\
&= \sum_{i=0}^m(-1)^i(2^{m+1-i}-1)\,\omega^{m-i}\int_X c_i \cdot L^{m-i}-(\omega-1)\sum_{i=0}^{m-1}(-1)^i(2^{m-i}-1)\,\omega^{m-1-i}\int_{X\cap Q}c_i(X\cap Q)\cdot\widetilde{L}^{m-1-i}\\
&= \sum_{i=0}^m(-1)^i(2^{m+1-i}-1)\,\omega^{m-i}\int_X c_i \cdot L^{m-i}-2(\omega-1)\sum_{i=0}^{m-1}(-1)^i(2^{m-i}-1)\,\omega^{m-1-i}\sum_{j=0}^i(-1)^{i-j}2^{i-j}\int_X c_j \cdot L^{m-j}\\
&= \sum_{i=0}^m(-1)^i\omega^{m-i}\left[(2^{m+1-i}-1)\int_X c_i \cdot L^{m-i}-\frac{2(\omega-1)}{\omega}(2^{m-i}-1)\sum_{j=0}^i(-2)^{i-j}\int_X c_j \cdot L^{m-j}\right]\,.
\end{aligned}$}
\]
The coefficient of $\int_X c_i\cdot L^{m-i}$ in the last line of the previous equation is
\[
\resizebox{0.95\textwidth}{!}{
$\begin{aligned}
&(-1)^i\omega^{m-i}(2^{m+1-i}-1)-\sum_{j=i}^m(-1)^j\omega^{m-j}\frac{2(\omega-1)}{\omega}(2^{m-j}-1)(-2)^{j-i}\\
&=(-1)^i\left[\omega^{m-i}(2^{m+1-i}-1)-\sum_{j=i}^m\left(\omega^{m-j}\frac{2(\omega-1)}{\omega}(2^{m-j}-1)\right)2^{j-i}\right]\\
&=(-1)^i\left[\omega^{m-i}(2^{m+1-i}-1)-\sum_{j=i}^m\left(\omega^{m-j}\frac{2(\omega-1)}{\omega}2^{m-i}\right)+\sum_{j=i}^m\left(\omega^{m-j}\left(\frac{\omega-2}{\omega}\right)2^{j-i}\right)+\sum_{j=i}^m\omega^{m-j}2^{j-i}\right]\\
&=(-1)^i\left[\omega^{m-i}(2^{m+1-i}-1)-\frac{\omega^{m-i+1}-1}{\omega}2^{m-i+1}-\frac{2^{m-i+1}-\omega^{m-i+1}}{\omega}+\sum_{j=i}^m\omega^{m-j}2^{j-i}\right]\\
&=(-1)^i\sum_{j=i}^m\omega^{m-j}2^{j-i}\,.
\end{aligned}$}
\]
For $\omega=2$, we have $(-1)^i\sum_{j=i}^m\omega^{m-j}2^{j-i}=(-1)^i(m+1-i)2^{m-i}$, while for $\omega\geq 3$
\[
(-1)^i\sum_{j=i}^m\omega^{m-j}2^{j-i}=(-1)^i\sum_{j=0}^{m-i}\omega^{m-j-i}2^{j}=(-1)^i\frac{\omega^{m+1-i}-2^{m+1-i}}{\omega-2}\,.
\]
In both cases, we get the desired identity \eqref{eq: RR degree X nonsingular given line bundle}.
\end{proof}

\begin{remark}
One might derive similar RR degree formulas as in Theorem \ref{thm: RR degree X nonsingular transversal to Q} for singular varieties: indeed \cite[Proposition 2.9]{aluffi2018projective} generalizes \eqref{eq: gen ED degree Chern} to singular varieties and \eqref{eq: Chern classes intersection with hypersurface} can also be applied to singular varieties by replacing Chern classes with Chern-Mather classes.
However, a complete proof of Theorem~\ref{thm: RR degree X nonsingular transversal to Q} for singular varieties requires a generalization of \cite[Theorem 1.5]{maxim2020defect} and \cite[Corollary 3.6]{maxim2020defect} to singular varieties.
\end{remark}

The following \verb|Macaulay2| function \cite{grayson1997macaulay2} computes the RR degree in \eqref{eq: RR degree X nonsingular given line bundle} for any input $m$, as a polynomial in $L,c_1,\ldots,c_m$, with parameter $\omega$:

\begin{tcolorbox}[width=\linewidth,colback=blue!5!white,colframe=blue!75!black]
\begin{Verbatim}[fontsize=\small,commandchars=\\\{\}]
RRdeg = m -> (
    c_0 = 1;
    R = \textcolor{purple}{QQ}[om][c_1..c_m,L, \textcolor{red}{Degrees} => \textcolor{blue}{toList}(1..m)|\{1\}];
    \textcolor{cyan}{return} \textcolor{blue}{sum}(m+1,i-> (-1)^i*\textcolor{blue}{sum}(m+1-i,j-> om^(m-i-j)*2^j)*c_i*L^(m-i))
    )

RRdeg(2) \textcolor{orange}{-- = c_2+(-om-2)*c_1*L+(om^2+2*om+4)*L^2}
\end{Verbatim}
\end{tcolorbox}

\begin{example}[RR degrees of nonsingular curves]\label{ex: RR degree X nonsingular curve}
Consider the assumptions of Theorem \ref{thm: RR degree X nonsingular transversal to Q}, and let $m=1$. One verifies that
\begin{equation}\label{eq: RR degree X nonsingular curve}
    \RRdeg_\omega(X) = \int_X (\omega+2)L-c_1\,.
\end{equation}
For example, let $X$ be the image of the Veronese embedding $\nu_d\colon\PP^1\hookrightarrow\PP^d$, after a generic linear transformation. We already know by Proposition \ref{prop: RR degree generic morphism} that $\RRdeg_\omega(X) = (\omega+2)d-2$.
One can derive the same identity using \eqref{eq: RR degree X nonsingular curve}. Let $h=\ko_{\PP^1}(1)$. In this case $c(\PP^1)=1+2h$ and $L=dh$, hence $\RRdeg_\omega(X) = \int_{\PP^1}(\omega+2)dh-2h=[(\omega+2)d-2]\int_{\PP^1}h=(\omega+2)d-2$.

Suppose instead that $X\subseteq\PP^2$ is a nonsingular curve of degree $\delta$.
We know already from Corollary \ref{corol: BW degree hypersurface} that $\RRdeg_\omega(X)=\delta(\delta+\omega-1)$.
We confirm this formula using \eqref{eq: RR degree X nonsingular curve}.
By \eqref{eq: total Chern class hypersurface} we have $c(X)=1+c_1=(1+L)^3(1-\delta L)=(1+3L)(1-\delta L)=1+(3-\delta)L$. Therefore $\RRdeg_\omega(X)=\int_X(\omega+2)L+(\delta-3)L=[\delta+\omega-1]\int_X L=\delta(\delta+\omega-1)$.

Alternatively, one can easily derive the same formulas by applying directly \eqref{eq: RR degree X nonsingular transversal to Q}, since in this case $\nu_\omega(X\cap Q)$ consists of $2\deg(X)$ simple points.\hfill$\diamondsuit$
\end{example}

\begin{example}[RR degrees of nonsingular surfaces]\label{ex: RR degree X nonsingular surface}
Consider the assumptions of Theorem \ref{thm: RR degree X nonsingular transversal to Q}, and let $m=2$. One verifies that
\begin{equation}\label{eq: RR degree X nonsingular surface}
    \RRdeg_\omega(X) = \int_X (\omega^2+2\omega+4)L^2-(\omega+2)c_1\cdot L+c_2\,.
\end{equation}
For example, let $X$ be the image of the Veronese embedding $\nu_d\colon\PP^2\hookrightarrow\PP(\mathrm{Sym}^d\C^3)$, after a generic linear transformation. We already know by Proposition \ref{prop: RR degree generic morphism} that $\RRdeg_\omega(X) = (\omega^2+2\omega+4)d^2-3(\omega+2)d+3$.
One can derive the same identity using \eqref{eq: RR degree X nonsingular surface}. Let $h=\ko_{\PP^2}(1)$. In this case $c(\PP^2)=1+3h+3h^2$ and $L=dh$, hence
\[
\RRdeg_\omega(X) = \int_{\PP^2}(\omega^2+2\omega+4)d^2h^2-(\omega+2)3dh^2+3h^2 = (\omega^2+2\omega+4)d^2-3(\omega+2)d+3\,,
\]
which confirms the previous formula.

Now let $X\subseteq\PP^3$ be a nonsingular surface of degree $\delta$. We know already from Corollary \ref{corol: BW degree hypersurface} that $\RRdeg(X)=\delta[(\delta-1)^2+\omega\delta+(\omega-1)^2]$.
Using \eqref{eq: total Chern class hypersurface}, we have $c(X)=1+c_1+c_2=(1+4L+6L^2)(1-\delta L+\delta^2L^2)=1+(4-\delta)L+(6-4\delta+\delta^2)L^2$, hence
\begin{align*}
\RRdeg_\omega(X) &= \int_X (\omega^2+2\omega+4)L^2-(\omega+2)(4-\delta)L^2+(6-4\delta+\delta^2)L^2\\
&= \delta[(\omega^2+2\omega+4)-(\omega+2)(4-\delta)+(6-4\delta+\delta^2)]\\
&= \delta[(\delta-1)^2+\omega\delta+(\omega-1)^2]\,,
\end{align*}
which confirms the previous formula.\hfill$\diamondsuit$
\end{example}

\begin{example}[RR degrees of nonsingular toric varieties]\label{ex: RR degree toric}
Let $\iota: X\hookrightarrow \PP^n$ be an $m$-dimensional toric embedding with corresponding lattice polytope $P\subseteq \R^{m}$. The Chern classes of $X$ are (see \cite[Proposition 13.1.2]{cox2011toric} and \cite[Corollary 11.5]{danilov1978geometry})
\[
    c_i=\sum_{F\in P_i}[F]\text{ with }\int_X c_i\cdot L^{m-i}=\sum_{F\in P_i}\mathrm{Vol}(F)\,,
\]
where $P_i$ is the set of faces of $P$ of codimension $i$ and $\mathrm{Vol}(F)$ is the normalized volume of $F$ with respect to the lattice induced by $\Z^{m}$ on the minimal linear space containing $F$.
As in the previous examples, we can use the formula of Theorem \ref{thm: RR degree X nonsingular transversal to Q} to compute the Rayleigh-Ritz degree of $X$, assuming it also satisfies the assumptions of the theorem:
\begin{equation}\label{eq: RR degree toric}
    \RRdeg_{\omega}(X)=\sum_{i=0}^m(-1)^i\left(\sum_{j=0}^{m-i}\omega^{m-i-j}2^j\right)\sum_{F\in P_i}\mathrm{Vol}(F).
\end{equation}
For example, consider the toric variety $X$ associated with the lattice polytope $P$ generated by the matrix
\[
    M=\begin{pmatrix} 1&0&2&1&2\\0&1&1&2&2\end{pmatrix}\,.
\]
Then $P$ is a nonsingular lattice polytope; therefore, the toric variety $X$ is also nonsingular. Suppose that $X$ is in general position with respect to the fixed quadric $Q$. Since it has six vertices and six edges, all with normalized volume $1$, and a single face of normalized volume $6$, we have
\begin{equation*}
    \RRdeg_{2}(X)=6\left(\sum_{i=0}^2(-1)^i(3-i)2^{2-i}\right)=54\,.\tag*{\text{$\diamondsuit$}}
\end{equation*}
\end{example}

In the following, we consider the tuples $\bm=(m_1,\ldots,m_k)\in\N^k$ and $\bd=(d_1,\ldots,d_k)\in\N^k$. Consider the multiprojective space $\PP^\bm\coloneqq\prod_{i=1}^k\PP^{m_i}$. Let $\nu_\bd\colon\PP^\bm\to\PP^M=\PP(\bigotimes_{i=1}^k\mathrm{Sym}^{d_i}\C^{m_i+1})$ be the Segre-Veronese embedding of $\PP^\bm$ via the line bundle $\ko_{\PP^\bm}(\bd)$. In the following, we denote its image $\nu_\bd(\PP^\bm)$ simply by $X$. We write $m=m_1+\cdots+m_k=\dim X$.
For all $i$, we denote by $\pi_i$ the projection of $\PP^\bm$ onto its $i$th factor $\PP^{m_i}$. Let $h_i\coloneqq c_1(\pi_i^*\ko_{\PP^{m_i}}(1))$ be the pullback of the hyperplane class in the $i$-th factor. Then by K\"unneth's formula the Chow ring of $\PP^\bm$ is $A^*(\PP^\bm)=\Z[h_1,\ldots,h_k]/(h_1^{m_1+1},\ldots,h_k^{m_k+1})$. We furthermore write $L=c_1(\ko_{\PP^\bm}(1))$ for the pullback of a hyperplane section from $\PP^M$. Tracking the pullback of $\ko_{\PP^M}(1)$ under the Segre-Veronese embedding $\nu_\bd$ we get $L=d_1h_1+\cdots+d_kh_k$.
Note that $L^m = (d_1h_1+\cdots+d_kh_k)^m = \frac{m!}{m_1!\cdots m_k!}(d_1h_1)^{m_1}\cdots(d_kh_k)^{m_k}$, in particular the degree of $\nu_\bd(\PP^\bm)$ is $\frac{m!}{m_1!\cdots m_k!}d_1^{m_1}\cdots d_k^{m_k}$.
More generally, we recall the following computations.

\begin{proposition}{\cite[Proposition 3.4]{kozhasov2023minimal}}\label{prop: degrees Chern classes Segre-Veronese}
Let $X$ be the image of the Segre-Veronese embedding $\nu_\bd$. Write $L=d_1h_1+\cdots+d_kh_k$ and $c(X)=1+c_1+\cdots+c_m$, where $m=m_1+\cdots+m_k$. Then
\begin{equation}\label{eq: degree ith Chern class Segre-Veronese}
\int_{\PP^\bm} c_i\cdot L^{m-i} = (m-i)!\sum_{\balpha\in\kp_{\bm,i}}\bd^{\bm-\balpha}\prod_{j=1}^k\frac{\binom{m_j+1}{\alpha_j}}{(m_j-\alpha_j)!}\quad\forall\,i\in\{0,\ldots,m\}\,,
\end{equation}
where $\bd^{\bm-\balpha}=d_1^{m_1-\alpha_1}\cdots d_k^{m_k-\alpha_k}$ and
$\kp_{\bm,i} \coloneqq \{\balpha\in\N^k\mid\text{$|\balpha|=i$ and $\alpha_j\le m_j$ for all $j\in[k]$}\}$.
\end{proposition}

The following \verb|Macaulay2| function computes the RR degree of a Segre-Veronese variety for any input $(\bm,\bd,\omega)$. In the code, \verb|m| and \verb|d| are lists with the components of $\bm$ and $\bd$ respectively, while \verb|om| is the integer $\omega\ge 1$:

\begin{tcolorbox}[width=\linewidth,colback=blue!5!white,colframe=blue!75!black]
\begin{Verbatim}[fontsize=\scriptsize,commandchars=\\\{\}]
RRdegSegVer = (m,d,om) -> (
    k = #m;
    dimX = \textcolor{blue}{sum}(m);
    R = \textcolor{purple}{QQ}[c_1..c_dimX,L,h_1..h_k,
	\textcolor{red}{Degrees}=>\textcolor{blue}{toList}(1..dimX)|\textcolor{blue}{toList}(k+1:1)]/\textcolor{blue}{ideal}(\textcolor{blue}{apply}(k, i-> h_(i+1)^(m#i+1)));
    c_0 = 1;
    totChernClassX = \textcolor{blue}{product}(k, i-> (1+h_(i+1))^(m#i+1));
    chernClassesX = \textcolor{blue}{apply}(dimX+1, i-> \textcolor{blue}{part}(i,totChernClassX));
    RRdeg = \textcolor{blue}{sum}(dimX+1, i-> (-1)^i*\textcolor{blue}{sum}(dimX+1-i, j-> om^(dimX-i-j)*2^j)*c_i*L^(dimX-i));
    subs = \textcolor{blue}{apply}(dimX, i-> c_(i+1) => chernClassesX#(i+1))|\{L=>\textcolor{blue}{sum}(k, i-> (d#i)*h_(i+1))\};
    \textcolor{cyan}{return} \textcolor{blue}{sub}(\textcolor{blue}{sub}(RRdeg, subs), \textcolor{blue}{apply}(k, i-> h_(i+1)=>1))
    )

RRdegSegVer(\{2,3\},\{2,2\},3) \textcolor{orange}{-- = 117240}
\end{Verbatim}
\end{tcolorbox}

In particular, we have the following corollary of Theorem \ref{thm: RR degree X nonsingular transversal to Q}, which coincides with \cite[Theorem 3.5]{kozhasov2023minimal} for $\omega=1$.

\begin{corollary}\label{corol: RR degree Segre-Veronese general position}
Let $X$ be the image of the Segre-Veronese embedding $\nu_\bd$, in general position with respect to the fixed isotropic quadric $Q$. Write $L=d_1h_1+\cdots+d_kh_k$ and $c(X)=1+c_1+\cdots+c_m$, where $m=m_1+\cdots+m_k$. Then
\begin{equation}\label{eq: RR degree Segre-Veronese general position}
\RRdeg_\omega(X) = \sum_{i=0}^m(-1)^i\left(\sum_{j=0}^{m-i}\omega^{m-i-j}2^j\right)(m-i)!\sum_{\balpha\in\kp_{\bm,i}}\bd^{\bm-\balpha}\prod_{j=1}^k\frac{\binom{m_j+1}{\alpha_j}}{(m_j-\alpha_j)!}
\end{equation}
where the index set $\kp_{\bm,i}$ is defined in Proposition \ref{prop: degrees Chern classes Segre-Veronese}.
\end{corollary}

\begin{example}\label{ex: RR degree product lines transversal Q}
The formula in Corollary \ref{corol: RR degree Segre-Veronese general position} simplifies when $\bm=\bd=\bone=(1,\ldots,1)\in\N^k$:
\begin{equation}\label{eq: RR degree product lines transversal Q}
    \RRdeg_\omega(X) = 
    \begin{cases}
    k!\,2^k\sum_{i=0}^k\frac{(-1)^i}{i!}(k+1-i) & \text{if $\omega=2$}\\[2pt]
    k!\sum_{i=0}^k\frac{(-1)^i}{i!}\frac{\omega^{k+1-i}2^i-2^{k+1}}{\omega-2} & \text{if $\omega\neq 2$.}
    \end{cases}
\end{equation}
We verify the formula above. In this case, we have $m=k$. Furthermore, the second sum in \eqref{eq: RR degree Segre-Veronese general position} is equal to $\binom{k}{i}2^i$ for all $i\in\{0,\ldots,m\}$. Hence
\[
    \RRdeg_\omega(X) = \sum_{i=0}^k(-1)^i\left(\sum_{j=0}^{k-i}\omega^{k-i-j}2^j\right)(k-i)!\binom{k}{i}2^i = k!\sum_{i=0}^k\frac{(-1)^i}{i!}\left(\sum_{j=0}^{k-i}\omega^{k-i-j}2^{i+j}\right)\,,
\]
which coincides with \eqref{eq: RR degree product lines transversal Q}.
For $\omega=1$, the formula provides the generic distance degree of a product of projective lines, given also in \cite[Corollary 3.8]{kozhasov2023minimal}.
One verifies that, for $k=2$, the degree is $\RRdeg_\omega(X)=2(\omega^2+2)$, which agrees with Example \ref{ex: RR degree X nonsingular surface} as $X$ is a degree 2 hypersurface in $\PP^3$ in this case.
For $\omega=2$ and $k\in\{2,\ldots,6\}$, the values of $\RRdeg_2(X)$ are $12,88,848,9888,135616$, respectively.\hfill$\diamondsuit$
\end{example}

Segre-Veronese embeddings are also special toric embeddings. Therefore, the formula in \eqref{eq: RR degree Segre-Veronese general position} can also be obtained by applying \eqref{eq: RR degree toric}.
For example, considering the variety $X=\sigma_k((\PP^1)^{\times k})$, the corresponding polytope is the $k$-dimensional hypercube. Since it has $2^i\binom{k}{i}$ faces of codimension $i$ all with normalized volume $(k-i)!$, using \eqref{eq: RR degree toric} we again obtain Equation~\eqref{eq: RR degree product lines transversal Q}.

\section{Rayleigh-Ritz degrees and singular tuples of tensors}\label{sec: RR degrees singular tuples}

Let $X$ be the image of the Segre-Veronese embedding $\nu_\bd$, in particular $X\subseteq\PP(V)$, where $V=\bigotimes_{i=1}^k\mathrm{Sym}^{d_i}\C^{m_i+1}$.
In this section, we fix the quadric $Q=Q_\mBW\subseteq\PP(V)$ and we compute $\RRdeg_\omega(X)$. In this case, $X$ is not transversal to $Q_\mBW$; therefore, we cannot apply Corollary \ref{corol: RR degree Segre-Veronese general position}. Indeed, under these assumptions, the RR degree of $X$ is smaller and is computed in the following result.

\begin{proposition}\label{prop: RR degree Segre-Veronese special}
For any $\bm\in\N^k$ and $\bd\in\N^k$, let $V=\bigotimes_{i=1}^k\mathrm{Sym}^{d_i}\C^{m_i+1}$ and consider the Segre-Veronese variety $X=\nu_\bd(\PP^\bm)\subseteq\PP(V)$. Fix the Bombieri-Weyl isotropic quadric $Q_\mBW\subseteq\PP(V)$. For any $\omega\ge 1$, $\RRdeg_\omega(X)$ equals the coefficient of the monomial $h_1^{m_1}\cdots h_k^{m_k}$ in the expansion of
\[
\prod_{i=1}^k\frac{\widehat{h}_i^{m_i+1}-h_i^{m_i+1}}{\widehat{h}_i-h_i},\quad\widehat{h}_i\coloneqq \omega\left(\sum_{j=1}^k d_jh_j\right)-h_i\,.
\]
\end{proposition}
\begin{proof}
Consider the notation used in Definition \ref{def: critical}. We recall a general property of the distance degree of degenerate varieties, namely, varieties whose linear span is not the entire ambient space. Consider a vector space $W$ equipped with a real positive-definite quadratic form $q$ and a degenerate affine variety $Z\subseteq W$. Let $\pi\colon W\to\langle Z\rangle$ be the $q$-orthogonal linear projection. Then, given a data point $u\in W$, we have $\mathrm{Crit}(\dist_{q,u}^2,Z)=\mathrm{Crit}(\dist_{q',\pi(u)}^2,Z)$, where $q'$ is the restriction of $q$ to $\langle Z\rangle$. For more details, see \cite[Lemma 2.1]{kubjas2022exact}. In particular, $\mathrm{DD}(Z,Q)=\mathrm{DD}(Z,Q')$, where $Q=\V(q)\subseteq\PP(W)$ and $Q'=\V(q')\subseteq\PP(\langle Z\rangle)$ respectively. Recall that $V=\bigotimes_{i=1}^k\mathrm{Sym}^{d_i}\C^{m_i+1}$.

Let us apply the previous fact in the setting of this proposition. Fix $V=\bigotimes_{i=1}^k\mathrm{Sym}^{d_i}\C^{m_i+1}$ and $W=\mathrm{Sym}^\omega V$ equipped with the Bombieri-Weyl quadratic form $q=q_\mBW$. Let $Z$ be the affine cone over $\nu_\omega(X)=\nu_\omega\circ\nu_\bd(\PP^\bm)$, in particular $\langle Z\rangle=\bigotimes_{i=1}^k\mathrm{Sym}^{\omega d_i}\C^{m_i+1}$ and
\begin{equation}\label{eq: codim W}
\codim_{W}\langle Z\rangle = \dim W-\dim\langle Z\rangle = \binom{\omega-1+\prod_{i=1}^k\binom{m_i+d_i}{d_i}}{\omega}-\prod_{i=1}^k\binom{m_i+\omega d_i}{\omega d_i}\,,
\end{equation}
hence $Z$ is degenerate for all $\omega>1$. In this case, the restriction $q'$ of $q_\mBW$ to $\langle Z\rangle$ is the Bombieri-Weyl quadratic form on $\bigotimes_{i=1}^k\mathrm{Sym}^{\omega d_i}\C^{m_i+1}$ associated with the Segre-Veronese embedding $\nu_{\omega\bd}=\nu_{(\omega d_1,\ldots,\omega d_k)}$, which we denote again with $q_\mBW$. By the previous fact, $\RRdeg_\omega(X)=\mathrm{DD}(\nu_{\omega\bd}(\PP^\bm),Q_\mBW)$.

By \cite[Theorem 1]{friedland2014number}, for any $\bm=(m_1,\ldots,m_k)$ and $\be=(e_1,\ldots,e_k)$ in $\N^k$, the distance degree $\mathrm{DD}(\nu_\be(\PP^\bm),Q_\mBW)$ equals the coefficient of the monomial $h_1^{m_1}\cdots h_k^{m_k}$ in the expansion of
\[
\prod_{i=1}^k\frac{\widehat{h}_i^{m_i+1}-h_i^{m_i+1}}{\widehat{h}_i-h_i},\quad\widehat{h}_i\coloneqq\left(\sum_{j=1}^k e_jh_j\right)-h_i\,.
\]
Therefore, the invariant $\RRdeg_\omega(X)$ is computed by applying the previous result for $\be=\omega\bd$.
\end{proof}

In the following corollary, we apply Proposition \ref{prop: RR degree Segre-Veronese special} in the case of a Segre product of $k$ projective lines. The resulting value is much smaller than the RR degree computed in \eqref{eq: RR degree product lines transversal Q}, where in that case we were assuming transversality between the Segre variety and the isotropic quadric.
As a side remark, when $\bm=\bd=\bone$ and $\omega=2$, the codimension in \eqref{eq: codim W} simplifies to $2^{k-1}(2^k+1)-3^k$. The corresponding integer sequence also appears in \cite[\href{http://oeis.org/A016269}{A016269}]{oeis}.

\begin{corollary}\label{corol: RR degree special product projective lines}
Assume $\bm=\bd=\bone=(1,\ldots,1)\in\N^k$ in Proposition \ref{prop: RR degree Segre-Veronese special}. Then
\[
    \RRdeg_\omega(X) = \omega^k k!\,.
\]
\end{corollary}
\begin{proof}
Here $\RRdeg_\omega(X)$ is the coefficient of $h_1\cdots h_k$ in the expansion of 
\[
\prod_{i=1}^k\frac{\widehat{h}_i^{m_i+1}-h_i^{m_i+1}}{\widehat{h}_i-h_i} = \prod_{i=1}^k \omega(h_1+\cdots+h_k) = \omega^k(h_1+\cdots+h_k)^k\,,
\]
that is precisely $\omega^kk!$. This concludes the proof.
\end{proof}

\begin{corollary}\label{corol: RR degree special rank-one matrices}
Assume $k=2$ and $\bd=(1,1)$ in Proposition \ref{prop: RR degree Segre-Veronese special}. Then
\[
    \RRdeg_\omega(X) = \sum_{i=1}^{\min\{m_1,m_2\}+1}\left[\sum_{\ell=0}^{m_1-i+1}\sum_{s=0}^{m_2-i+1}\binom{\ell+i-1}{\ell}\binom{s+i-1}{s}(\omega-1)^{\ell+s}\right]\omega^{2(i-1)}\,.
\]
In particular, when $\omega=2$, 
\[
\RRdeg_2(X) = \sum_{i=1}^{\min\{m_1,m_2\}+1}\binom{m_1+1}{i}\binom{m_2+1}{i}4^{i-1}\,.
\]
\end{corollary}
\begin{proof}
We start noticing that
\begin{align*}
&\frac{((\omega-1)h_1+\omega h_2)^{m_1+1}-h_1^{m_1+1}}{(\omega-2)h_1+\omega h_2}=\sum_{i=0}^{m_1}\sum_{j=0}^i\binom{i}{j}(\omega-1)^{i-j}\omega^{j}h_1^{m_1-j}h_2^{j}\\
&=\sum_{j=0}^{m_1}\left(\sum_{i=j}^{m_1}\binom{i}{j}(\omega-1)^{i-j}\right)\omega^{j}h_1^{m_1-j}h_2^{j}=\sum_{j=0}^{m_1}\left(\sum_{\ell=0}^{m_1-j}\binom{\ell+j}{j}(\omega-1)^\ell\right)\omega^{j}h_1^{m_1-j}h_2^{j}\,. 
\end{align*}
Applying Proposition \ref{prop: RR degree Segre-Veronese special} with $k=2$, $\bd=(1,1)$, we need to extract the coefficient of $h_1^{m_1}h_2^{m_2}$ in the expansion of
\[
\frac{((\omega-1)h_1+\omega h_2)^{m_1+1}-h_1^{m_1+1}}{(\omega-2)h_1+\omega h_2}\cdot \frac{(\omega h_1+(\omega-1)h_2)^{m_2+1}-h_2^{m_2+1}}{\omega h_1+(\omega-2)h_2}\,.    
\]
By the first equality, we obtain
\begin{align*}
\RRdeg_\omega(X) &= \sum_{j=0}^{\min\{m_1,m_2\}}\left(\sum_{\ell=0}^{m_1-j}\binom{\ell+j}{j}(\omega-1)^\ell\right)\left(\sum_{s=0}^{m_2-j}\binom{s+j}{j}(\omega-1)^s\right)\omega^{2j}\\
&= \sum_{i=1}^{\min\{m_1,m_2\}+1}\left[\sum_{\ell=0}^{m_1-i+1}\sum_{s=0}^{m_2-i+1}\binom{\ell+i-1}{\ell}\binom{s+i-1}{s}(\omega-1)^{\ell+s}\right]\omega^{2(i-1)}\,.
\end{align*}
Finally, when $\omega=2$, the previous equality can be simplified using the following well-known equality
\[
    \sum_{j=0}^n\binom{j+p-1}{j}=\binom{n+p}{n}\,.\qedhere
\]
\end{proof}

Next, we keep the setting of Proposition \ref{prop: RR degree Segre-Veronese special} and highlight the correspondence among the critical points of the constrained Rayleigh quotient optimization problem \eqref{eq: main optimization} and the normalized singular tuples of a $k$-homogeneous polynomial associated with the data symmetric tensor $f$.
Given tuples $\bm$, $\bd$, and $\be$ in $\N^k$, we identify $\bigotimes_{i=1}^k\mathrm{Sym}^{e_i}\C^{m_i+1}$ with $\C[x]_\bd\coloneqq\bigotimes_{i=1}^k\C[x_i]_{d_i}$, the space of $k$-homogeneous polynomials of multidegree $\bd$, where $x_i=(x_{i,0},\ldots,x_{i,m_i})$ for all $i\in[k]$. We write each element $f\in\C[x]_\bd$ as
\[
f = \sum_{\alpha_1,\ldots,\alpha_k} c_{\alpha_1\cdots\alpha_k}\prod_{i=1}^k\binom{d_i}{\alpha_i}x_i^{\alpha_i}
\]
for some coefficients $c_{\alpha_1\cdots\alpha_k}\in\C$, where the sum ranges over all $(\alpha_1,\ldots,\alpha_k)\in\prod_{i=1}^k\N^{m_i+1}$ such that $|\alpha_i|=d_i$ for all $i\in[k]$.

\begin{definition}\label{def: singular tuples}
Let $f\in\C[x]_\bd$. A {\em (normalized) singular $k$-tuple of $f$} is a tuple $(v_1,\ldots,v_k)$ of vectors $v_i\in\C^{m_i+1}$ such that
\[
\sum_{j=0}^{m_i}v_{i,j}^2=1\quad\text{and}\quad\frac{1}{d_i}\nabla_i f(v_1,\ldots,v_k) = \sigma\,v_i\quad\forall\,i\in[k]
\]
for some $\sigma\in\C$, where $\nabla_i f$ is the gradient vector of $f$ with respect to the $i$th vector of coordinates. The value $\sigma$ is the {\em (normalized) singular value} associated with $(v_1,\ldots,v_k)$.
\end{definition}

In projective space, the singular $k$-tuples of $f$ are the fixed points of the $k$-gradient map
\[
\nabla f\colon\PP^\bm \dashrightarrow \PP^\bm\,,\quad\nabla f([x_1],\ldots,[x_k])\coloneqq([\nabla_i f(x_1,\ldots,x_k)])_{i=1}^k\,.
\]

\begin{proposition}\label{prop: critical points and singular tuples}
Consider the assumptions of Proposition \ref{prop: RR degree Segre-Veronese special}
and the projection $\pi\colon\mathrm{Sym}^\omega V\to\bigotimes_{i=1}^k\mathrm{Sym}^{\omega d_i}\C^{m_i+1}$.
For any $f\in\mathrm{Sym}^\omega V$, the locus of normalized $X$-eigenvectors of $f$ coincides with the locus of normalized singular tuples of $\pi(f)\in\mathrm{Sym}^{\omega d_i}\C^{m_i+1}=\C[x]_{\omega\bd}$.
\end{proposition}
\begin{proof}
Following the argument in the proof of Proposition \ref{prop: RR degree Segre-Veronese special}, for any $f\in\mathrm{Sym}^\omega V$, the locus of normalized $X$-eigenvectors of $f$ coincides with the locus of critical points of the Bombieri-Weyl distance function from $\pi(f)$ restricted to $C(\nu_{\omega\bd}(\PP^\bm))\subseteq\C[x]_{\omega\bd}$. With a similar argument as in the proof of Proposition \ref{prop: equivalence}, one can verify that the latter is precisely the locus of normalized singular $k$-tuples of $\pi(f)$, see also \cite[\S 7]{friedland2014number}.
\end{proof}

In the following, we describe in more detail the projection $\pi$ of Proposition \ref{prop: critical points and singular tuples}. For the sake of brevity, we choose $\bd=\bone\in\N^k$ and $\omega=2$. A similar study can be carried out for any $\bd\in\N^k$ and $\omega\ge 2$. First, we consider a small example.

\begin{example}
The first nontrivial case holds for $\bn=(1,1)$. Let $X=\nu_{(1,1)}(\PP^1\times\PP^1)\subseteq\PP(\C^2\otimes\C^2)$ and consider its Veronese embedding $\nu_2(X)\subseteq\PP(\mathrm{Sym}^2(\C^2\otimes\C^2))\cong\PP^9$. The projective span of $\nu_2(X)$ is $\PP(\mathrm{Sym}^2\C^2\otimes\mathrm{Sym}^2\C^2)\cong\PP^8$, a hyperplane in $\PP(\mathrm{Sym}^2(\C^2\otimes\C^2))\cong\PP^9$. In coordinates, we write a symmetric matrix $H\in\mathrm{Sym}^2(\C^2\otimes\C^2)$ as
\[
H =
\begin{pmatrix}
    h_{(0,0),(0,0)} & h_{(0,0),(0,1)} & h_{(0,0),(1,0)} & h_{(0,0),(1,1)} \\
    h_{(0,0),(0,1)} & h_{(0,1),(0,1)} & h_{(0,1),(1,0)} & h_{(0,1),(1,1)} \\
    h_{(0,0),(1,0)} & h_{(0,1),(1,0)} & h_{(1,0),(1,0)} & h_{(1,0),(1,1)} \\
    h_{(0,0),(1,1)} & h_{(0,1),(1,1)} & h_{(1,0),(1,1)} & h_{(1,1),(1,1)}
\end{pmatrix}\,.
\]
We identify $\PP(\mathrm{Sym}^2\C^2\otimes\mathrm{Sym}^2\C^2)$ as the hyperplane in $\PP(\mathrm{Sym}^2(\C^2\otimes\C^2))$ of equation $h_{(0,0),(1,1)}=h_{(0,1),(1,0)}$. Then the projection $\pi\colon \mathrm{Sym}^2(\C^2\otimes\C^2)\to\mathrm{Sym}^2\C^2\otimes\mathrm{Sym}^2\C^2$ maps $H$ to the bi-homogeneous polynomial
\[
f_H = \sum_{|\alpha_1|=|\alpha_2|=2}\binom{2}{\alpha_1}\binom{2}{\alpha_2}c_{\alpha_1,\alpha_2}x_1^{\alpha_1}x_2^{\alpha_2}\,,
\]
where
\[
\begin{gathered}
c_{(2,0)(2,0)} = h_{(0,0),(0,0)}\,,\ c_{(2,0)(1,1)} = h_{(0,0),(0,1)}\,,\ c_{(2,0)(0,2)} = h_{(0,1),(0,1)}\,,\\
c_{(1,1)(2,0)} = h_{(0,0),(1,0)}\,,\ c_{(1,1)(1,1)} = \frac{1}{2}(h_{(0,0),(1,1)}+h_{(0,1),(1,0)})\,,\ c_{(1,1)(0,2)} = h_{(0,1),(1,1)}\,,
\\
c_{(0,2)(2,0)} = h_{(1,0),(1,0)}\,,\ c_{(0,2)(1,1)} = h_{(1,0),(1,1)}\,,\ c_{(0,2)(0,2)} = h_{(1,1),(1,1)}\,.
\end{gathered}
\]
Applying Proposition \ref{prop: RR degree Segre-Veronese special}, $\RRdeg_2(\PP^1\times\PP^1)$ equals the coefficient of the monomial $h_1h_2$ in the expansion of $4(h_1+h_2)^2$, which is $8$. Therefore, a generic symmetric matrix $H\in\mathrm{Sym}^2(\C^2\otimes\C^2)$ gives $8$ pairwise distinct critical points of the Rayleigh quotient optimization problem \eqref{eq: main optimization} constrained to $X=\nu_{(1,1)}(\PP^1\times\PP^1)$, corresponding to the $8$ singular pairs of $f_H$.\hfill$\diamondsuit$
\end{example}

To generalize the previous example, we introduce the following notation.
For every $\bm=(m_1,\ldots,m_k)$, define $I_\bm\coloneqq\prod_{\ell=1}^k\{0,\ldots,m_\ell\}$. We consider $\{z_\bi\}_{\bi\in I_\bm}$ as a set of coordinates of $V=\bigotimes_{\ell=1}^k\C^{m_\ell+1}$. Let $\preceq$ denote the lexicographic order in $I_\bm$. We define also the set $\{h_{\bi,\bj}\}_{\bi\preceq\bj\in I_\bm}$ of variables of $\mathrm{Sym}^2V$, where $h_{\bi,\bj}\coloneqq h_{\bj,\bi}$ if $\bj\preceq\bi$. We define the relation $(\bi,\bj)\sim(\bi',\bj')$ if and only if
\[
\left(\text{{\tt\#} $0$\,s in $(i_\ell,j_\ell)$},\ldots,\text{{\tt\#} $m_\ell$\,s in $(i_\ell,j_\ell)$}\right)=\left(\text{{\tt\#} $0$\,s in $(i'_\ell,j'_\ell)$},\ldots,\text{{\tt\#} $m_\ell$\,s in $(i'_\ell,j'_\ell)$}\right)\quad\forall\,\ell\in[k]\,.
\]
One verifies that this is an equivalence relation. We denote by $\alpha(\bi,\bj)=(\alpha_1(\bi,\bj),\ldots,\alpha_k(\bi,\bj))\in\N^{|I_\bm|}$ the tuple where
\begin{equation}\label{eq: relations multiindices alpha}
\alpha_\ell(\bi,\bj) \coloneqq \left(\text{{\tt\#} $0$s in $(i_\ell,j_\ell)$},\text{{\tt\#} $1$s in $(i_\ell,j_\ell)$},\ldots,\text{{\tt\#} $m_\ell$s in $(i_\ell,j_\ell)$}\right)\quad\forall\,\ell\in[k]\,.
\end{equation}
The tuple $\alpha=(\alpha_1,\ldots, \alpha_k)$ encodes exactly the equivalence classes of the relation $\sim$. We denote by $p(\alpha_1,\ldots,\alpha_k)$ the set of all pairs $(\bi,\bj)$ with $\alpha(\bi,\bj)=\alpha$.

\begin{proposition}\label{prop: correspondence critical points Rayleigh and singular tuples}
Consider the space $V=\bigotimes_{\ell=1}^k\C^{m_\ell+1}$ and the variety $X=\nu_\bone(\PP^\bm)\subseteq\PP(V)$.
The projective span $\langle\nu_2(X)\rangle\subseteq\PP(\mathrm{Sym}^2 V)$ is defined by the equations $h_{\bi,\bj} = h_{\bi',\bj'}$ whenever $(\bi,\bj)\sim(\bi',\bj')$.

For any symmetric matrix $H=(h_{\bi,\bj})_{\bi\preceq\bj\in I_\bm}\in\mathrm{Sym}^2V$, the critical points of the optimization problem \eqref{eq: main optimization} with respect to the Bombieri-Weyl isotropic quadric in $\PP(V)$ correspond to the singular $k$-tuples of the $k$-homogeneous polynomial $f_H=(c_{\alpha_1\cdots\alpha_k})_{\alpha_\ell\in\N^{m_\ell+1},|\alpha_\ell|=2}\in\C[x]_{\boldsymbol{2}}=\C[x]_{(2,\ldots,2)}$ where
\[
c_{\alpha_1\cdots\alpha_k}=\frac{1}{|p(\alpha_1,\ldots,\alpha_k)|}\sum_{(\bi,\bj)\in p(\alpha_1,\ldots,\alpha_k)}h_{\bi,\bj}\,.
\]
\end{proposition}
\begin{proof}
Given a symmetric matrix $H\in\mathrm{Sym}^2V$, by Proposition \ref{prop: critical points and singular tuples} the critical points of the squared Bombieri-Weyl distance function from $H$ are the same as the critical points of the projection $\pi(H)$ of $H$ onto $\bigotimes_{\ell=1}^k\mathrm{Sym}^2\C^{m_\ell+1}$. The restriction of the Bombieri-Weyl quadratic form on $\mathrm{Sym}^2V$ induced by the standard Euclidean inner product on $V$ coincides with the Bombieri-Weyl quadratic form on $\bigotimes_{\ell=1}^k\mathrm{Sym}^2\C^{m_\ell+1}$ induced by the standard Euclidean inner products on the factors $\C^{m_\ell+1}$, for all $\ell\in[k]$. Furthermore, the critical points of the Bombieri-Weyl distance function from a partially symmetric tensor $f\in\bigotimes_{\ell=1}^k\mathrm{Sym}^2\C^{m_\ell+1}=\C[x]_{\boldsymbol{2}}$ to the cone over $\nu_{\boldsymbol{2}}(\PP^\bm)$ correspond to the normalized singular $k$-tuples of $f$.

The Veronese embedding $\nu_2$ maps the point $[\{z_\bi\}_{\bi\in I_\bm}]\in\PP(V)$ to the point $[\{z_\bi z_\bj\}_{\bi\preceq\bj\in I_\bm}]\in\PP(\mathrm{Sym}^2V)$. After identifying each variable $z_\bi$ with the monomial $\prod_{\ell=1}^k x_{\ell,i_\ell}$ in the Segre embedding $\nu_\bone\colon\PP^\bm\hookrightarrow\PP(V)$, then $z_\bi z_\bj = z_{\bi'}z_{\bj'}$ for every $(\bi,\bj)\sim(\bi',\bj')$.
Each of these quadratic relations in the $z_\bi$'s corresponds to a linear relation $h_{\bi,\bj} = h_{\bi',\bj'}$ in the variables of $\PP(\mathrm{Sym}^2V)$, namely these equations define the projective span $\langle\nu_2(X)\rangle=\PP(\bigotimes_{\ell=1}^k\mathrm{Sym}^2\C^{m_\ell+1})$ as a subspace of $\PP(\mathrm{Sym}^2V)$. From this, we define the orthogonal projection (with respect to the Bombieri-Weyl quadratic form in $\mathrm{Sym}^2V$)
\[
\pi\colon\mathrm{Sym}^2V\to\bigotimes_{\ell=1}^k\mathrm{Sym}^2\C^{m_\ell+1}\,,\quad H=(h_{\bi,\bj})_{\bi\preceq\bj\in I_\bm}\mapsto f_H=(c_{\alpha_1\cdots\alpha_k})_{\alpha_\ell\in\N^{m_\ell+1},|\alpha_\ell|=2}\,,
\]
by averaging over the equivalence classes of $\sim$ as in the statement of the proposition.
\end{proof}

We conclude this section with a discussion on real $X$-eigenvectors of real homogeneous polynomials.
In \cite{kozhasov2018fully}, Kozhasov showed that, for any degree $\omega$, there exists a real polynomial $f\in\R[x]_\omega$ all of whose eigenpoints are real and pairwise distinct, and their number is equal to $\mathrm{DD}(\nu_\omega(\PP^n),Q_\mBW)$. In the case of ternary forms, such examples can be constructed via real line arrangements, see \cite[Theorem 6.1]{abo2017eigenconfigurations}. In the following example, we exhibit a real projective variety $X\subseteq\PP^n$ such that, for every $f\in\R[x]_\omega$, if $f$ admits $\RRdeg_\omega(X)$ pairwise distinct eigenpoints, then not all of them are real.
This implies that $\RRdeg_\omega(X)$ is not always a sharp upper bound for the number of real $X$-eigenvectors of a real homogeneous polynomial $f\in\R[x]_\omega$.

\begin{example}\label{ex: no six real}
Consider the plane conic $X\subseteq\PP^2$ of equation $x_1^2-x_0x_2=0$, image of the Veronese embedding $\nu_2\colon\PP^1\hookrightarrow\PP^2$ defined by $\nu_2([t_0:t_1])=[t_0^2:t_0t_1:t_1^2]$.
Let $Q\subseteq\PP^2$ be the isotropic quadric of equation $q(x)=x_0^2+x_1^2+x_2^2=0$, then $X\cap Q$ consists of four distinct points, hence $X$ is in general position with respect to $Q$. Following up on Example \ref{ex: RR degree X nonsingular curve}, we have $\RRdeg_2(X)=6$. We show that, given a homogeneous polynomial
\[
f=c_{200}x_0^2+2\,c_{110}x_0x_1+2\,c_{101}x_0x_2+c_{020}x_1^2+2\,c_{011}x_1x_2+c_{002}x_2^2\in\R[x_0,x_1,x_2]_2\,,
\]
if $f$ admits $6$ pairwise distinct $X$-eigenpoints, then either two or four of them are real, but never six. Using Lemma \ref{lem: polynomial map}, the six $X$-eigenpoints correspond, via $\nu_2$, to the roots $[t_0:t_1]\in\PP^1$ of the binary sextic
\begin{equation}\label{eq: det 2x2 matrix}
\det
\begin{pmatrix}
    \frac{\partial(f\circ\nu_2)}{\partial t_0} & \frac{\partial(f\circ\nu_2)}{\partial t_1}\\[2pt]
    \frac{\partial(q\circ\nu_2)}{\partial t_0} & \frac{\partial(q\circ\nu_2)}{\partial t_1}
\end{pmatrix}
= 8\cdot p(t_0,t_1)\,,
\end{equation}
where
\[
\resizebox{\textwidth}{!}{
$\begin{aligned}
    p = \underbrace{c_{110}}_{y_0}t_0^6+\underbrace{(c_{020}+2c_{101}-c_{200})}_{y_1}t_0^5t_1+\underbrace{(3c_{011}-c_{110})}_{y_2}t_0^4t_1^2+\underbrace{2(c_{002}-c_{200})}_{y_3}t_0^3t_1^3+\underbrace{(c_{011}-3c_{110})}_{y_4}t_0^2t_1^4+\underbrace{(c_{002}-c_{020}-2c_{101})}_{y_5}t_0t_1^5\underbrace{-c_{011}}_{y_6}t_1^6\,.
\end{aligned}$}
\]
We checked that the coefficients $y_0,\ldots,y_6$ of $p$ satisfy three independent linear relations
\begin{equation}\label{eq: linear conditions}
3\,y_2-y_4+8\,y_6 = 2\,y_1-y_3+2\,y_5 = 3\,y_0+y_4+y_6 = 0\,.
\end{equation}
Working modulo the ideal generated by the linear polynomials in \eqref{eq: linear conditions}, we verified that the discriminant of a binary sextic $\sum_{i=0}^6y_it_0^{6-i}t_1^i$, a homogeneous polynomial of degree $10$ in the coefficients $y_0,\ldots,y_6$, specializes to the product of the following three polynomials (up to a scalar factor):
\begin{align}\label{eq: polynomials g1 g2 g3}
\begin{split}
    g_1 &= \left(15\,y_3-4\,y_4-24\,y_5+56\,y_6\right)^2+27\left(y_{3}-4\,y_{4}+4\,y_{5}\right)^2\\
    g_2 &= \left(15\,y_3+4\,y_4-24\,y_5+56\,y_6\right)^2+27\left(y_{3}+4\,y_{4}+4\,y_{5}\right)^2\\
    g_3 &= 864\,y_{3}^{2}y_{4}^{2}y_{5}^{2}+6\,084\,y_{4}^{4}y_{5}^{2}-3\,456\,y_{3}^{3}y_{5}^{3}-29\,376\,y_{3}y_{4}^{2}y_{5}^{3}+20\,736\,y_{3}^{2}y_{5}^{4}+13\,824\,y_{4}^{2}y_{5}^{4}\\
    &\quad-41\,472\,y_{3}y_{5}^{5}+27\,648\,y_{5}^{6}-3\,456\,y_{3}^{2}y_{4}^{3}y_{6}-24\,336\,y_{4}^{5}y_{6}+15\,552\,y_{3}^{3}y_{4}y_{5}y_{6}\\
    &\quad+129\,672\,y_{3}y_{4}^{3}y_{5}y_{6}-93\,312\,y_{3}^{2}y_{4}y_{5}^{2}y_{6}-60\,912\,y_{4}^{3}y_{5}^{2}y_{6}+196\,992\,y_{3}y_{4}y_{5}^{3}y_{6}\\
    &\quad-228\,096\,y_{4}y_{5}^{4}y_{6}-23\,328\,y_{3}^{4}y_{6}^{2}-164\,268\,y_{3}^{2}y_{4}^{2}y_{6}^{2}+14\,352\,y_{4}^{4}y_{6}^{2}+143\,856\,y_{3}^{3}y_{5}y_{6}^{2}\\
    &\quad+8\,028\,y_{3}y_{4}^{2}y_{5}y_{6}^{2}-296\,352\,y_{3}^{2}y_{5}^{2}y_{6}^{2}+542\,736\,y_{4}^{2}y_{5}^{2}y_{6}^{2}+241\,920\,y_{3}y_{5}^{3}y_{6}^{2}-117\,504\,y_{5}^{4}y_{6}^{2}\\
    &\quad+42\,444\,y_{3}^{2}y_{4}y_{6}^{3}-327\,376\,y_{4}^{3}y_{6}^{3}-772\,776\,y_{3}y_{4}y_{5}y_{6}^{3}+525\,312\,y_{4}y_{5}^{2}y_{6}^{3}+326\,457\,y_{3}^{2}y_{6}^{4}\\
    &\quad-87\,312\,y_{4}^{2}y_{6}^{4}-1\,143\,936\,y_{3}y_{5}y_{6}^{4}+408\,384\,y_{5}^{2}y_{6}^{4}-984\,624\,y_{4}y_{6}^{5}-1\,263\,376\,y_{6}^{6}\,.
\end{split}
\end{align}
Let $W\subseteq\PP(\R[t_0,t_1]_6)$ denote the three-dimensional real projective subspace defined by the linear polynomials in \eqref{eq: linear conditions}. We are interested in the complement of the real hypersurface in $W$ cut out by the polynomial $g_1g_2g_3$. In each connected component, the number of real roots $[t_0:t_1]$ of a binary sextic is constant.
Since $g_1$ and $g_2$ are sums of two distinct squares, the complement of the zero locus cut out by $g_1g_2$ is connected. Therefore, it is enough to study the connected components of the complement of the hypersurface cut out by $g_3$. We performed this computation using the \verb|Julia| \cite{bezanson2017julia} package \verb|HypersurfaceRegions|. For details on the package, see \cite{breiding2025computing}. After declaring variables and the polynomial $g_3$, the command \verb|regions()| returns the output displayed in Figure \ref{fig: regions}, confirming that there are exactly two connected regions in the complement of $g_3$ in $W$. The real zero locus of $g_3(y_3,y_4,y_5,y_6)$ in the chart $\{y_6=1\}$ is displayed in Figure \ref{fig: plot}. In the two connected regions of the complement, there are either two or four distinct real solutions. For example, we verified that starting from the real ternary forms
\begin{align*}
f_1 &= 42\,x_{0}^{2}+280\,x_{0}x_{1}+267\,x_{1}^{2}-120\,x_{0}x_{2}+420\,x_{1}x_{2}-63\,x_{2}^{2}\\
f_2 &= 189\,x_{0}^{2}+56\,x_{0}x_{1}+195\,x_{1}^{2}+120\,x_{0}x_{2}-168\,x_{1}x_{2}+147\,x_{2}^{2}
\end{align*}
the corresponding binary sextics, up to a constant factor, are respectively
\begin{align*}
    p_1 &= 4\,t_{0}^{6}+3\,t_{0}^{5}t_{1}+14\,t_{0}^{4}t_{1}^{2}-6\,t_{0}^{3}t_{1}^{3}-6\,t_{0}^{2}t_{1}^{4}-6\,t_{0}t_{1}^{5}-6\,t_{1}^{6}\\
    p_2 &= 2\,t_{0}^{6}+9\,t_{0}^{5}t_{1}-20\,t_{0}^{4}t_{1}^{2}-6\,t_{0}^{3}t_{1}^{3}-12\,t_{0}^{2}t_{1}^{4}-12\,t_{0}t_{1}^{5}+6\,t_{1}^{6}\,.
\end{align*}
For $j\in\{1,2\}$, the roots of $p_j$ in $\PP^1$ are of the form $[1:\alpha]$ with $\alpha\in A_j$, where
\begin{align*}
    A_1 &= \{-0.673903,\,1.04536,\,-0.058973\pm0.730365\,i,\,-0.501755\pm1.92718\,i\}\\
    A_2 &= \{-6.08441,\,-0.807225,\,0.34297,\,2.04914,\,-0.000239\pm0.932267\,i\}\,,
\end{align*}
in particular $f_1$ and $f_2$ have respectively two and four distinct real $X$-eigenvectors, and the previous argument shows that it is not possible to find a real ternary quadric $f$ with $6$ pairwise distinct real $X$-eigenvectors.\hfill$\diamondsuit$
\end{example}

\begin{figure}[!htb]
    \centering
    \begin{minipage}{0.45\textwidth}
        \centering
        \includegraphics[width=\textwidth]{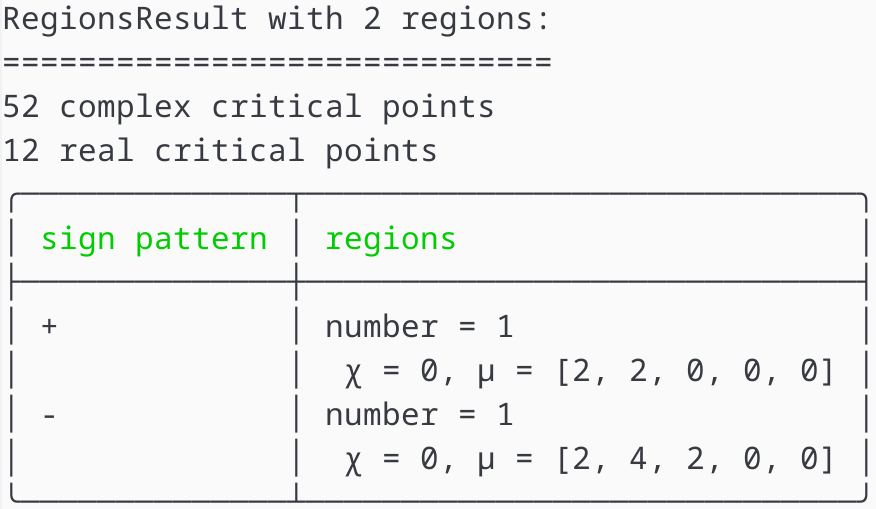}
        \caption{The output of {\tt HypersurfaceRegions.jl}.}
        \label{fig: regions}
    \end{minipage}%
    \begin{minipage}{0.55\textwidth}
        \centering
        \begin{overpic}[width=0.9\textwidth]{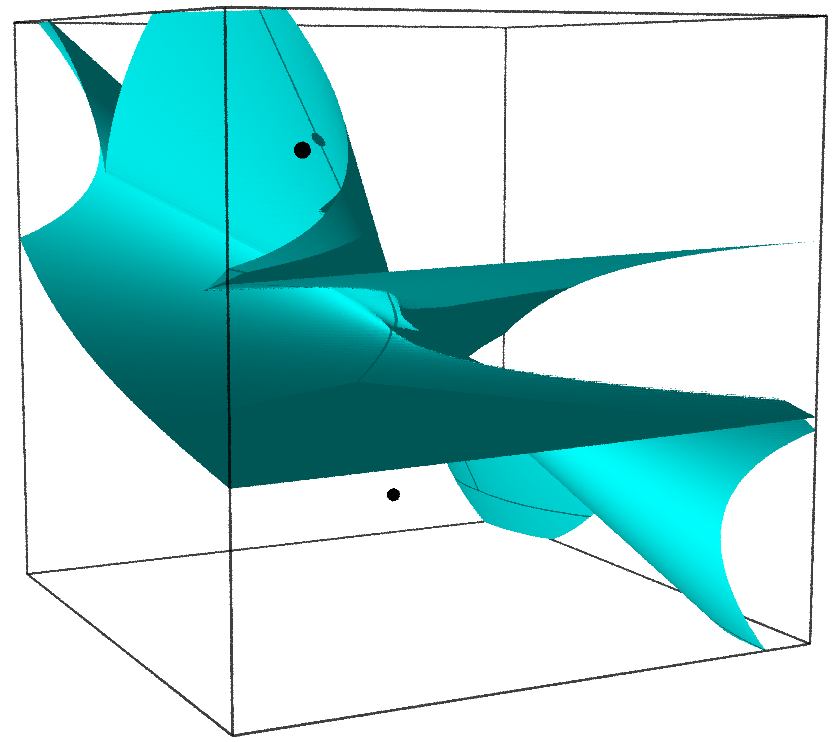}
            \put (31,74) {$p_1$}
            \put (50,28) {$p_2$}
        \end{overpic}
        \caption{The real zero locus of the polynomial $g_3$ in \eqref{eq: polynomials g1 g2 g3}.}
        \label{fig: plot}
    \end{minipage}
\end{figure}

In the previous example, we considered the standard Euclidean quadratic form $q(x)=x_0^2+x_1^2+x_2^2$. Choosing instead the Bombieri-Weyl quadratic form $q_\mBW(x)=x_0^2+2\,x_1^2+x_2^2$, then $X$ is no longer in general position with respect to $Q$, indeed $X\cap Q$ is scheme-theoretically the union of two double points. More precisely, we are in the situation of Proposition \ref{prop: RR degree Segre-Veronese special} for $k=1$, $m_1=1$, $d_1=2$, and $\omega=2$. The Rayleigh-Ritz degree in this case is $4$. Furthermore, applying Proposition \ref{prop: critical points and singular tuples}, the $X$-eigenvectors of $f\in\mathrm{Sym}^2(\mathrm{Sym}^2\C^2)$ are the eigenvectors of the binary quartic $\pi(f)\in\mathrm{Sym}^4\C^2$, and we know by the results in \cite{kozhasov2018fully} that it is possible to find a generic real binary quartic with exactly four distinct normalized eigenvectors. This translates to the interesting property that, in this case, the polynomial $p(t_0,t_1)$ in \eqref{eq: det 2x2 matrix} factors as $p=(t_0^2+t_1^2)\cdot p'$, where
\[
\resizebox{\textwidth}{!}{
$\begin{aligned}
p' = c_{110}t_{0}^{4}+(c_{020}-2\,c_{200}+2\,c_{101})t_{0}^{3}t_{1}+3(c_{011}-c_{110})t_{0}^{2}t_{1}^{2}+(2\,c_{002}-c_{020}-2\,c_{101})t_{0}t_{1}^{3}-c_{011}t_{1}^{4}
\end{aligned}$.}
\]
The first factor $t_0^2+t_1^2$ is one on the Bombieri-Weyl unit sphere because $q_\mBW\circ\nu_2(t_0,t_1)=(t_0^2+t_1^2)^2$. With a similar computation as in Example \ref{ex: no six real}, we verified that the factor $p'$ can have either two or four distinct real roots in $\PP^1$ for any choice of real parameters $c_{ijk}$.

\section*{Acknowledgements}

We thank Viktoriia Borovik, Hannah Friedman, Serkan Ho\c{s}ten, Khazhgali Kozhasov, Max Pfeffer, and M\'{a}t\'{e} L. Telek for the fruitful discussions and the valuable feedback received.
We would also like to thank Bernd Sturmfels for suggesting the idea of the project.
F. S. is supported by the P500PT-222344 SNSF project. J. W. is supported by the SPP 2458 ``Combinatorial Synergies'', funded by the Deutsche Forschungsgemeinschaft (DFG, German Research Foundation), project ID: 539677510.

\bibliographystyle{alphaurl}
\bibliography{biblio}

\end{document}